\documentclass[12pt]{amsart}
\usepackage{amssymb,enumerate,mathrsfs, amsmath, amsthm}
\usepackage{url,titletoc,enumitem}
\usepackage{amsmath,amsthm,amscd,amssymb,latexsym,verbatim,cite}

\usepackage[usenames,dvipsnames]{xcolor}
\definecolor{darkblue}{rgb}{0.0, 0.0, 0.55}
\usepackage[pagebackref,colorlinks,linkcolor=BrickRed,citecolor=OliveGreen,urlcolor=darkblue,hypertexnames=true]{hyperref}
\usepackage{cmap}

\renewcommand{\subset}{\subseteq}
\renewcommand{\emptyset}{\varnothing}

\newtheorem{theorem}{Theorem}[section]
\newtheorem{cor}[theorem]{Corollary}
\newtheorem{lemma}[theorem]{Lemma}

\newtheorem*{lemma*}{Lemma}

\def\beq{\begin{equation}}
\def\eeq{\end{equation}}
\def\lb{\label}

\numberwithin{equation}{section}

\def\beq{\begin{equation}}
\def\eeq{\end{equation}}

\def\bbR{ {\mathbb R}}
\def\bbN{ {\mathbb N}}

\def\bbD{ {\mathbb D}}
\def\bbC{ {\mathbb C}}
\def\bbZ{ {\mathbb Z}}
\def\calT{ {\mathcal T}}
\def\calS{ {\mathcal S}}
\def\calR{ {\mathcal R}}
\def\calI{ {\mathcal I}}
\def\calJ{ {\mathcal J}}
\def\tht{\theta}
\def\tbet{\tilde\beta}

\def\eps{\epsilon}

\def\beq{\begin{equation}}
\def\eeq{\end{equation}}

\def\sgn{{\rm sgn}}

\title{
 Euler Equations on General Planar Domains}
\author{Zonglin Han and Andrej Zlato\v s}
\address{\noindent Department of Mathematics \\ UC San Diego \\ La Jolla, CA 92093, USA \newline Email:
zoh003@ucsd.edu, zlatos@ucsd.edu}

\begin{document}

\begin{abstract} 
We obtain a general sufficient condition on the geometry of possibly singular planar domains that guarantees global uniqueness for any weak solution to the Euler equations on them whose vorticity is bounded and initially constant near the boundary.  This condition is only slightly more restrictive than exclusion of  corners with angles greater than $\pi$ and, in particular, is satisfied by all convex domains.  The main ingredient in our approach is showing that constancy of the vorticity near the boundary is preserved for all time because Euler particle trajectories on these domains, even for general bounded solutions, cannot reach the boundary in finite time.  We then use this to show that no vorticity can be created by the boundary of such possibly singular domains for general bounded solutions.  We also show that our condition is essentially sharp in this sense by constructing domains that come arbitrarily close to satisfying it, and on which particle trajectories can reach the boundary in finite time.  In addition, when the condition is satisfied, we find sharp bounds on the asymptotic rate of the fastest possible approach of particle trajectories to the boundary.\end{abstract}

	\maketitle

\section{Introduction}

The study of motions of incompressible inviscid fluids, in mathematics, physics, as well as engineering,  is both a centuries old endeavor and a vibrant area of current research.  Mathematically, these motions are modeled by the Euler equations
\begin{align} 
 \partial_t u + (u\cdot \nabla) \,u &= -\nabla p,  \label{1.0a}
 \\  \nabla \cdot u & = 0,  \label{1.0b}
\end{align}
with $u$ the fluid velocity and $p$ its pressure.   These PDE are usually considered for times $t>0$ and on spatial domains $\Omega\subseteq\bbR^d$ with impermeable boundaries and hence with the {\it no-flow} (or {\it slip}) boundary condition 
 \begin{equation} \label{1.0}
u\cdot n = 0
 \end{equation}
on $\bbR^+\times\partial\Omega$, with $n$ the unit outer normal to $\Omega$.  
Despite the immense variety of advances in the area since Euler's formulation of this simple looking but  incredibly rich system of PDE in 1755, some of the most important questions about its solutions remain open to this day.  While the most famous of these is the question of finite time singularity of solutions in three and more dimensions, even in two spatial dimensions there are several important unsolved problems.  One of these is uniqueness of solutions on irregular domains --- even just general convex ones --- due to singular effects of rough boundaries on the dynamics of fluids.  
 
In two dimensions, the case considered here, the Euler equations can be equivalently reformulated as the active scalar equation
\begin{equation}\label{1.1}
\partial_t  \omega + u\cdot \nabla \omega = 0 
\end{equation}
on $\bbR^+\times\Omega\subseteq \bbR^+\times\bbR^2$, with
\[
 \omega:=\nabla\times u= \partial_{x_1} u_{2} - \partial_{x_2} u_{1}
 \]
 the {\it vorticity} of the flow.  This conveniently removes the pressure from the system, and one can now also find the (divergence-free) velocity from the vorticity via the Biot-Savart law
\begin{equation}\label{1.2}
u=\nabla^\perp\Delta^{-1} \omega,
\end{equation}
with $\Delta$ the Dirichlet Laplacian on $\Omega$ and $\nabla^\perp\psi:=(-\partial_{x_2} \psi, \partial_{x_1} \psi)$.

\medskip
\noindent {\bf Prior Existence and Uniqueness Results}
\medskip

On smooth bounded domains $\Omega\subseteq\bbR^2$, global well-posedness for strong solutions goes back to the breakthrough 1933 papers by Wolibner \cite{Wol} and H\" older \cite{Hold}  (for unbounded domains, see \cite{McGrath,Kikuchi}).  A natural class of solutions to consider are those with bounded vorticities, due to \eqref{1.1} preserving $\|\omega(t,\cdot)\|_{L^\infty}$, and global well-posedness for weak solutions with initial conditions $\omega_0 \in L^\infty(\Omega)$ was proved in the celebrated work of Yudovich \cite{Yudth} (see also \cite{Bardos,MB,MP, Temam}).  
While existence of global weak solutions can also be proved for $\omega_0 \in L^p(\Omega)$ \cite{DiPernaMajda} and even for $\omega_0 \in H^{-1}(\Omega)\cap \mathcal M_+(\Omega)$ \cite{Delort}, uniqueness appears likely to not always hold in this case.  Indeed, this is suggested by recent results of Vishik  \cite{Vis1,Vis2}, who showed non-uniqueness of solutions on $\mathbb R^2$ with $\omega_0  \in L^p(\mathbb R^2)$ for some $p>2$, in the presence of a forcing from the same space.

The above results apply on sufficiently smooth domains, with  $\partial\Omega$ being $C^{1,1}$ or better.  However, global existence of (even unbounded) solutions has been proved to hold on much less regular domains.    Indeed, this was done via $L^2$ estimates on the velocity $u$ for $\omega_0 \in L^p(\Omega)$ or $\omega_0 \in H^{-1}(\Omega)\cap \mathcal M_+(\Omega)$ by Taylor on convex domains \cite{Taylor}, and later by G\' erard-Varet and Lacave for very general irregular domains \cite{GV-Lac, GV-Lac2}.

Low regularity of the boundary is, however, currently a crucial barrier to a resolution of the uniqueness of solutions question on general bounded domains, even for bounded solutions.  In a nutshell, all presently available uniqueness results require the velocity to be close to Lipschitz in an appropriate sense, and
sufficient smoothness of $\partial\Omega$ is typically needed to obtain apriori estimates on the Riesz transform $\nabla\nabla^\perp\Delta^{-1}\omega=\nabla u$.  This includes the approach of Yudovich, via the family of Calder\' on-Zygmund inequalities
\[
\|\nabla u(t,\cdot)\|_{L^p} \le Cp\|\omega(t,\cdot)\|_{L^p}
\]
for all $p\in[2,\infty)$ and with a uniform $C$, as well as the use of the {\it log-Lipschitz} estimate
\beq \lb{1.4}
\sup_{x,y\in\Omega} \frac {|u(t,x)-u(t,y)|}{|x-y| \max \{1,-\ln|x-y|\} }  \le C \|\omega(t,\cdot)\|_{L^\infty} 
\eeq
(see, e.g.,  \cite{MP}).  However, such estimates do not hold in general on less regular domains.  For instance, $\nabla u$ may only be $L^2$ near irregular portions of the boundaries
 of general convex domains (even for smooth $\omega$),
 while Jerison and Kenig showed that $\nabla u$ may not even be integrable on some $C^1$ domains \cite{JK}.  

It is therefore not surprising that uniqueness of all weak solutions has so far only been established for a fairly small class of (simply connected) non-$C^{1,1}$ domains.  In fact, all these must be $C^{1,1}$ except at finitely many corners with acute (including right) angles.  Specifically, this was achieved first for rectangles by Bardos, Di Plinio, and Temam \cite{BDT}, then for domains that are $C^{2,\gamma}$ (for some $\gamma>0$) except at finitely many acute corners by Lacave, Miot, and Wang \cite{LMW}, and then on domains that are $C^{1,1}$ except at finitely many acute corners by Di Plinio and Temam \cite{DT}.  In all these results, intersections of the domains with small discs centered at all corners were even assumed to be exact sectors. 
Corners with angles greater than $\frac\pi 2$ (and all other irregular geometries of $\partial\Omega$) are excluded in these results due to the velocity not being close to Lipschitz there even for smooth $\omega$ (at corners with angles greater than $\pi$, the velocity is in general even unbounded).  Uniqueness of general solutions outside of the class of domains from \cite{DT} therefore appears to be a very challenging open problem.

Nevertheless, 
one may still hope to establish uniqueness on irregular domains for solutions that remain constant in the regions where the velocity fails to be close to Lipschitz (similarly to results for the vortex-wave system \cite{LacMio,MarPul-VWS}, when the diffuse part of the vorticity remains constant near all point vortices).  This may mean neighborhoods of corners with angles greater than $\frac\pi 2$ for piecewise $C^{1,1}$ domains, or all of $\partial\Omega$ for general irregular domains.  In fact, since  Euler particle trajectories for bounded solutions starting inside {\it smooth} domains $\Omega$ cannot approach $\partial\Omega$ faster than double-exponentially in time, all solutions that are initially constant near all of $\partial\Omega$ will remain such for all later times.  One may hope that this property extends to many less regular domains, possibly with other asymptotic rates of approach to the boundary, which would yield  a large class of initial data on such domains with  unique global weak solutions.  

This approach was recently taken up by Lacave and the second author.  Lacave first proved in \cite{Lacave-SIAM} that if $\partial\Omega$ is $C^{1,1}$ except at finitely many corners that are all exact sectors with angles greater than $\frac\pi 2$, and $\omega_0 $ has a constant sign and is constant near $\partial\Omega$, then $\omega$ will indeed remain constant near $\partial\Omega$ forever and weak solutions are unique.  Then, together with Zlato\v s, they showed the same result when $\partial\Omega$ is $C^{1,1}$ except at finitely many corners of arbitrary angles from $(0,\pi)$ that do not need to be exact sectors, and without the  sign restriction on $\omega_0 $  \cite{LacZla}.  In both works, Euler particle trajectories for bounded solutions (general ones in \cite{LacZla} and with a constant sign in \cite{Lacave-SIAM}) were shown to remain in $\Omega$ for all time (again approaching $\partial\Omega$ no faster than double-exponentially), and   in \cite{LacZla} this was even proved to hold when $\partial\Omega$ is only $C^{1,\gamma}$ (for some $\gamma>0$) except at finitely many corners with angles from $(0,\pi)$.  Moreover, \cite{LacZla} also constructed examples of domains smooth everywhere except at a single corner with an arbitrary angle from $(\pi,2\pi)$ where Euler particle trajectories can reach $\partial\Omega$ in finite time, using an idea of Kiselev and Zlato\v s \cite{KisZla}.

\medskip
\noindent {\bf Definitions and Main Results}
\medskip

The uniqueness results in \cite{Lacave-SIAM,LacZla}, just as those in \cite{BDT,DT,LMW}, still require piecewise $C^{1,1}$ domains.  In the present paper we greatly expand this class by considering general {\it regulated} bounded Lipschitz domains, that is, those having a (counter-clockwise) {\it forward} tangent vector at each point of $\partial\Omega$ (see \eqref{1.7a} below), whose argument is a function with left and right limits everywhere. In particular, this includes all convex domains.  

We then obtain a general condition guaranteeing that Euler particle trajectories for bounded weak solutions in these domains never reach $\partial\Omega$, and also prove existence and uniqueness of global  weak solutions for all vorticities initially constant near $\partial\Omega$.  Our condition is only slightly more restrictive than exclusion of corners with angles greater than $\pi$, which was shown to be necessary in \cite{LacZla}, and it places no restrictions on those segments of $\partial\Omega$ where the argument of the forward tangent vector is non-decreasing (so, in particular, it is satisfied by all convex $\Omega$).  Specifically, our condition is satisfied precisely when the  argument of the forward tangent vector to $\partial\Omega$, composed with the Riemann mapping for $\Omega$, can be written as a sum of an arbitrary increasing function and a second function that has a modulus of continuity $m$ from  a precisely defined class  of moduli (which includes, e.g., $m$ with  $m(r)=\frac \pi{2|\log r|}$ for all small enough $r>0$).   Moreover, for any concave modulus $m$ from this class, we find the exact (up to a constant factor in time) asymptotic rate of the fastest possible approach of Euler particle trajectories to $\partial\Omega$ among all domains as above.  We also show that no vorticity can be created by the boundary of these possibly singular domains, a result that even extends in a weaker form to general bounded domains (see Corollary \ref{C.2.2}). 

Finally, we  show that our condition is essentially sharp.  Specifically, for each concave modulus not in the above class of moduli (e.g., $m$ with $m(r)=\frac a{2|\log r|}$ for all small enough $r>0$, with any fixed $a>\pi$), we construct a domain as above in which particle trajectories can reach the boundary in finite time.    It therefore appears that our work pushes right up to the limits of the philosophy from \cite{LacMio,Lacave-SIAM,LacZla,MarPul-VWS}, within the class of regulated domains at least, and further significant advances will likely require a breakthrough in  the question of uniqueness for solutions that are not constant near all those singular segments of $\partial\Omega$ where the Euler velocities corresponding to bounded vorticities may be far from Lipschitz.  Our Theorem \ref{T.1.1}(ii) and Corollary \ref{C.2.2} below represent a first step in this effort.

Let us now state the precise definitions and our main results.  Let $\Omega\subseteq\bbR^2$ be an open bounded Lipschitz domain with $\partial \Omega$ a Jordan curve, and let $\mathcal{T}:\Omega\to \mathbb{D}$ be a Riemann mapping (with $\bbD$ the unit disc in $\bbC=\bbR^2$).  By the  Kellogg-Warschawski Theorem (see, e.g., \cite[Theorem 3.6]{Pomm}), we can then extend $\calT$ continuously to $\bar\Omega$.  We also let $\calS:=\calT^{-1}$.  
We will consider here solutions to the Euler equations on $\Omega$ from the {\it Yudovich class}
\[
\left\{ (\omega,u)\in L^\infty \left( (0,\infty);L^\infty(\Omega)\times L^2(\Omega) \right) \big|\,\, \text{$\omega= \nabla \times u$, and \eqref{1.0b}--\eqref{1.0} hold weakly} \right\},
\]
where the weak form of \eqref{1.0b}--\eqref{1.0} is 
\[
\int_\Omega u(t,\cdot)\cdot\nabla h \, dx = 0 
\qquad \forall h\in H^1_{\rm loc}(\Omega) \text{ with $\nabla h\in L^2(\Omega)$}
\]
for almost all $t>0$ (see \cite{GV-Lac, GV-Lac2}).
Such $\omega$ and $u$ are then equivalently related by the Biot-Savart law \eqref{1.2}.  This can be expressed in terms of $\calT$ and the Dirichlet Green's function $G_\bbD(\xi,z)=\frac 1{2\pi}\ln\frac{|\xi-z|}{|\xi-z^*| |z|}$ for $\bbD$ (with $z^* := {z}{|z|^{-2}}$ and $(a,b)^\perp:=(-b,a)$) as
\beq\label{1.111}
 u(t,x) = \frac1{2\pi}D\mathcal{T}(x)^T\int_{\Omega} \left( \frac{ \mathcal{T}(x)- \mathcal{T}(y)}{| \mathcal{T}(x)- \mathcal{T}(y)|^2} - \frac{ \mathcal{T}(x)- \mathcal{T}(y)^*}{| \mathcal{T}(x)- \mathcal{T}(y)^*|^2 } \right)^\perp \omega(t,y)\, dy.
\eeq
Since $u$ is uniquely determined by $\omega$, we will simply say that $\omega$ is from the Yudovich class.

We say that $\omega$ from the Yudovich class is a {\it weak solution} to the Euler equations on $\Omega$, on time interval $(0,T)$ and  with initial condition $\omega_0\in L^\infty(\Omega)$, if
\beq\label{1.222}
\int_0^{T} \int_\Omega \omega \left(  \partial_t \varphi + u \cdot \nabla \varphi \right) \,dx dt= -\int_\Omega \omega_0 \varphi(0, \cdot) \,dx \qquad \forall \, \varphi \in C_0^\infty \left([0, T) \times \Omega\right).
\eeq
This is obviously the definition of weak solutions to the transport equation \eqref{1.1}, but it is also equivalent to the relevant weak velocity formulation of the Euler equations on $\Omega$ (see \cite[Remark 1.2]{GV-Lac2}).  When $T=\infty$ we call such solutions {\it global}.  Their existence is guaranteed by \cite{GV-Lac} for very general $\Omega$, but the question of uniqueness is still open in general.

It is well known (see, e.g., \cite[Chapter 2]{MP}) that uniform boundedness of $\omega$ shows that the velocity is locally log-Lipshitz, uniformly in time.  Specifically, \eqref{1.4} holds for all $t\in(0,\infty)$ with $\Omega$ replaced by any compact $K\subseteq\Omega$ and with $C=C_{\Omega,K}$.  Then $u$ is also uniformly-in-time locally bounded on $\Omega$, and for each $x\in\Omega$ there is a unique solution to the ODE
\beq \label{1.333}
\frac d{dt} X_t^x = u(t,X_t^x) \qquad\text{and}  \qquad X_0^x=x
\eeq
on an interval $(0,t_x)$ such that 
\[
t_x:=\sup\{t>0\,|\,X_s^x\in\Omega \text{ for all $s\in(0,t)$}\}
\]
 (so if $X_t^x$ reaches $\partial\Omega$, then $t_x$ is the first such time).  That is, $\{X_t^x\}_{t\in[0,t_x)}$ is the Euler particle trajectory for the particle starting at $x\in\Omega$.  We note that a priori the ODE only holds for almost all $t\in(0,t_x)$ (with $X_t^x$ continuous in time), but we will show that $u$ is continuous and therefore \eqref{1.333} holds for all $t\in[0,t_x)$
 (see Corollary \ref{C.2.2} below).

For any $\tht\in\bbR$, the {\it unit forward tangent vector} to $\Omega$ at $\calS(e^{i\tht})\in\partial\Omega$ is the unit vector 
\beq\lb{1.7a}
\bar\nu_\calT(\tht) := \lim_{\phi \rightarrow \tht+} \frac{\mathcal{S}(e^{i\phi})-\mathcal{S}(e^{i\tht})}  {|\mathcal{S}(e^{i\phi})-\mathcal{S}(e^{i\tht})|} ,
\eeq
provided this limit exists.   If it does for each $\tht\in\bbR$, and the limits $ \lim_{\phi \to \tht\pm} \bar\nu_\calT(\phi) $ both exists at each $\tht\in\bbR$, then the domain $\Omega$ is said to be {\it regulated}.  In this case obviously $\lim_{\phi \rightarrow \tht+} \bar\nu_\calT(\phi)=\bar\nu_\calT(\tht)$, while the argument of the complex number $ \bar\nu_\calT(\tht) \left[ \lim_{\phi \rightarrow \tht-} \bar\nu_\calT(\phi) \right]^{-1}$ equals $\pi$ minus the interior angle of $\Omega$ at $\calS(e^{i\tht})$.  We then let
\beq\lb{1.7b}
\bar \beta_\calT(\tht) := \bar {\rm arg}\, \bar\nu_\calT(\tht),
\eeq
where $\bar {\rm arg}$ is the argument of a complex number plus some integer multiple of $2\pi$. This multiple is chosen so that $\bar\beta_\calT(0)\in[0,2\pi)$ and  $\bar\beta_\calT(\tht)-\lim_{\phi \to \tht-}\bar\beta_\calT(\phi)\in[-\pi,\pi]$ for each $\tht\in\bbR$, and if $\Omega$ has cusps, we do it so that this difference is $\pi$ at exterior cusps (with interior angle 0) and $-\pi$ at interior cusps (with interior angle $2\pi$).  Of course, then this difference is again $\pi$ minus the interior angle of $\Omega$ at $\calS(e^{i\tht})$.  Since we only consider Lipschitz domains here (i.e., without cusps), we will  always have  $\bar\beta_\calT(\tht)-\lim_{\phi \to \tht-}\bar\beta_\calT(\phi)\in(-\pi,\pi)$.

The above defines the right-continuous function $\bar\beta_\calT:\bbR\to\bbR$ uniquely, and it satisfies $\bar\beta_\calT(\tht + 2\pi) = \bar\beta_\calT(\tht) + 2\pi$  for all $\tht \in \mathbb{R}$.   As we wrote above, whether Euler particle trajectories for bounded solutions can reach the boundary in finite time depends on how quickly is  $\bar\beta_\calT$ allowed to decrease locally (which happens when $\bar\nu_\calT$ turns clockwise), with no restrictions on its increase. 
This will be quantified in terms of a modulus of continuity for one of two components of  $\bar\beta_\calT$, with the other component being an arbitrary increasing function.

We call a function $m:[0,2\pi]\to[0,\infty)$ with $m(0)=0$  a {\it modulus} if it is continuous, non-decreasing, and satisfies $m(a+b)\le m(a)+m(b)$ for any $a,b\in[0,2\pi]$ with $a+b\le 2\pi$.  If some $f:\bbR\to\bbR$ satisfies $|f(\tht )-f(\phi)|\le m(r)$ for all $r\in[0,2\pi]$ and all $\tht ,\phi\in\bbR$ with $|\tht -\phi|\le r$, we say that $f$ has {\it modulus of continuity} $m$.  We also let 
\[
q_m(s):=s\exp \left( \frac 2\pi \int_{s}^1 \frac{m(r)}r dr \right),
\]
and if $\int_0^1 \frac {ds}{q_m(s)}=\infty$, we let $\rho_m:\bbR\to (0,1)$ be the inverse function to $y\mapsto  \ln \int_{y}^{1} \frac {ds}{ q_m(s)}$, so
\[
\rho_m \left( \ln \int_{y}^{1} \frac {ds}{ q_m(s)} \right) =y .
\]
Then $\rho_m$ is decreasing with $\lim_{t\to-\infty} \rho_m(t)=1$ and $\lim_{t\to\infty} \rho_m(t)=0$, and we shall see that it is the maximal asymptotic approach rate of Euler particle trajectories to $\partial\Omega$ (up to a constant factor in time) among all domains for which the first component of  $\bar\beta_\calT$ from the preceding paragraph has
modulus of continuity $m$.  
Note also that $\int_0^1 \frac {ds}{q_m(s)}=\infty$ holds whenever $\int_0^1 \frac{m(r)}r dr<\infty$, and functions with such moduli $m$ are called {\it Dini continuous}.

In our main results, we will assume the following hypothesis. 

\begin{itemize}
\item[{\rm\bf(H)}] 
Let $\Omega\subseteq\bbR^2$ be a regulated open bounded  Lipschitz domain with $\partial \Omega$ a Jordan curve.
Let $\mathcal{T}:\Omega\to \mathbb{D}$ be a Riemann mapping and let $\beta_\calT,\tilde\beta_\calT$ be functions on $\bbR$ with $2\pi$-periodic (distributional) derivatives 
such that $\beta_\calT$ is non-decreasing, $\tilde \beta_\calT$ has some modulus of continuity $m$ with $q_m$ and $\rho_m$ defined above, and  the argument of the (counter-clockwise) forward tangent vector to $\partial\Omega$ is $\bar\beta_\calT=\beta_\calT+\tilde\beta_\calT$.
\end{itemize}

Note that if $\beta_\calT,\tilde\beta_\calT$ are as above and their sum is the argument of the forward tangent vector to a Jordan curve $\partial\Omega$, then the bounded domain $\Omega$ must automatically be regulated.

As mentioned above, neither {\bf (H)} nor our results place any restrictions on $\beta_\calT$.  In particular, the following main result of the present paper holds for any convex domain $\Omega$, since then one can let $\beta_\calT:=\bar \beta_\calT$ and $\tilde\beta_\calT\equiv 0$ (and therefore  $m\equiv 0$).

\begin{theorem} \lb{T.1.1}
Assume {\bf (H)} and that $\int_0^1 \frac {ds}{q_m(s)}=\infty$.  Let $\omega_0 \in L^\infty(\Omega)$ and let $\omega$ from the Yudovich class be any global weak solution to the Euler equations on $\Omega$ with initial condition $\omega_0$
(such solutions are known to exist by \cite{GV-Lac}).

(i)
We have $t_x=\infty$ for all $x\in\Omega$, and for any $R<1$ and all large enough $t> 0$, 
\beq \lb{1.5}
\sup_{|\calT(x)|\le R} |\calT(X_t^x)|\le 1-\rho_m(500 \|\omega\|_{L^\infty} t)
\eeq
(except when $\omega\equiv 0$, but then $X_t^x\equiv x$). And if  $\tilde\beta_\calT$ is  Dini continuous,
then the right-hand side of \eqref{1.5} can be replaced by the $m$-independent bound $1- \exp({-e^{500 \|\omega\|_{L^\infty} t}})$.

(ii)
We have $\{X_t^x\,|\,x\in\Omega\}=\Omega$ for all $t>0$, and  $\omega(t,X_t^x)=\omega_0(x)$ for \hbox{a.e.~$(t,x)\in \bbR^+\!\times\Omega$.}  Moreover,  $u$ is continuous on $[0,\infty) \times\Omega$ and \eqref{1.333} holds for all $(t,x)\in [0,\infty)\times\Omega$. 

(iii)
If ${\rm supp}\,( \omega_0-a)\cap\partial\Omega = \emptyset$ for some $a\in\bbR$, then the solution $\omega$  is unique.
\end{theorem}

{\it Remarks.}  
1.  This naturally extends to solutions on time intervals $(0,T)$ for  $T\in(0,\infty)$.\smallskip

2.  Part (i) also shows that $\inf_{|\calT(x)|\le R}d(X_t^x,\partial\Omega)\ge \rho_m(500 \|\omega\|_{L^\infty} t)$ for any $R<1$, due to $\calT$ being H\"older continuous for Lipschitz $\Omega$ (see, e.g., \cite[Theorem 2]{Lesley}).
This is because our proof shows that (i) also holds with 499 in place of 500, and one can easily show that $\rho_m(500 c t)\le \frac 1N\rho_m(499 c t)^N$ for any fixed $c,N> 0$ and all large enough $t>0$.
\smallskip

3.  
A ``borderline'' case for the condition $\int_0^1 \frac {ds}{q_m(s)}=\infty$ is $m(r)=\frac a{|\log r|}$  for all small $r>0$ (with $a\ge 0$).  Here $\int_0^1 \frac {ds}{q_m(s)}=\infty$ holds precisely when $a\le \frac \pi 2$, while $\int_0^1 \frac{m(r)}r dr=\infty$ for all $a>0$.  In this case $\rho_m$ is still a double exponential when $a<\frac\pi 2$, as for Dini continuous  $\tilde\beta_\calT$, but a triple exponential when $a=\frac \pi 2$.  The double-exponential rate is known to be the maximal possible boundary approach rate for  {\it smooth} domains, due to \eqref{1.4} holding there, but \eqref{1.4} fails  even for general convex domains.
See also the remark after Theorem \ref{T.1.2} below.\smallskip



Our second main result, which applies to  concave moduli $m$, shows that Theorem \ref{T.1.1}(i) is essentially sharp, even for stationary solutions.  

\begin{theorem} \lb{T.1.2}
For any concave modulus $m$, there is a domain $\Omega$ satisfying {\bf (H)} and a stationary weak solution $\omega$ from the Yudovich class to the Euler equations on $\Omega$  such that the following hold.  

(i)
If $\int_0^1 \frac {ds}{q_m(s)}<\infty$, then $X_t^x\in \partial \Omega$ for some $x\in\Omega$ and $t>0$.

(ii)
If $\int_0^1 \frac {ds}{q_m(s)}=\infty$, then $|\calT(X_t^x)|\ge 1- \rho_m(c t)$ for some $x\in\Omega$, $c>0$, and all $t\ge 0$.
\end{theorem}

{\it Remark.}  
Note that if $m(r)= a(L_1(\frac 1r)\dots L_{k-1}(\frac 1r))^{-1} + \frac \pi 2 \sum_{j=1}^{k-2} (L_1(\frac 1r)\dots L_j(\frac 1r))^{-1}$  for all small enough $r>0$, with $k\ge 2$, $a\in[0,\frac\pi 2)$, and $L_j(r)$ being $\ln r$ composed $j$ times, then  $\rho_m$ is essentially a $k$-tuple exponential.  Therefore all such boundary approach rates do occur on some domains $\Omega$ to which Theorem \ref{T.1.1}(i) applies.
\smallskip


We also note that Theorem \ref{T.1.1} has a natural analog when the forward tangent vector is defined via arc-length parametrization of $\partial\Omega$, rather than via $\calS$.  If $\sigma:[0,2\pi]\to\partial\Omega$ is the (counter-clockwise) constant speed parametrization of $\partial\Omega$ (extended to be $2\pi$-periodic on $\bbR$, and obviously unique up to translation), 
then Lemma 1 in  \cite{War} shows that $\calT\circ \sigma$ and its inverse (modulo $2\pi$) are H\" older continuous.  If we therefore use
\beq\lb{1.7c}
\bar\nu_\Omega(\tht) := \lim_{\phi \rightarrow \tht+} \frac{\sigma(\phi)-\sigma(\tht)}  {|\sigma(\phi)-\sigma(\tht)|} ,
\eeq
instead of \eqref{1.7a},
and the corresponding $\tilde\beta_\Omega$ (with $\bar\beta_\Omega,\beta_\Omega,\tilde\beta_\Omega$ chosen analogously to $\bar\beta_\calT,\beta_\calT,\tilde\beta_\calT$) has some modulus of continuity $m$, then  $\tilde\beta_\calT$ has modulus of continuity $\tilde m(r):=m(Cr^\gamma)$ for some $C,\gamma>0$.
But since a simple change of variables shows that $\int_0^1 \frac{m(r)}r dr<\infty$ is equivalent to $\int_0^1 \frac{m(Cr^\gamma)}r dr<\infty$, we obtain the following result.

\begin{cor} \lb{C.1.4}
Theorem \ref{T.1.1}
 continues to hold when \eqref{1.7a} and $\bar\beta_\calT,\beta_\calT,\tilde\beta_\calT$ in {\bf (H)} are replaced by \eqref{1.7c} and $\bar\beta_\Omega,\beta_\Omega,\tilde\beta_\Omega$, respectively, and if  $\tilde\beta_\Omega$ is also Dini continuous.
\end{cor}

{\it Remarks.}
1.  Of course, while $\bar\beta_\calT,\beta_\calT,\tilde\beta_\calT$ depend on $\calT$, they can also be made to only depend on $\Omega$ because we are free to choose $\calT$.
\smallskip

2.  Note that if an open bounded simply connected Lipschitz domain $\Omega$ can be touched from the outside  by a disc of uniform radius at each point of $\partial\Omega$ (i.e., $\Omega$ satisfies the {\it uniform exterior sphere condition}), and we replace \eqref{1.7a}  by \eqref{1.7c}, then these hypotheses are satisfied with $m(r)=Cr$ for some constant $C$.  Hence Corollary \ref{C.1.4}  
holds for all such domains.

Finally, we provide here a version of Theorem \ref{T.1.1}(ii) for general open bounded domains, which follows from its proof and is also  of independent interest.  To the best of our knowledge, such results previously required $\partial\Omega$ to be at least $C^{1,1}$ (see, e.g., \cite{Lacave-SIAM, LMW, LacZla}).

\begin{cor} \lb{C.2.2}
Let $\omega$ from the Yudovich class be a weak solution  to the Euler equations on an open bounded domain $\Omega\subseteq\bbR^2$, on time interval $(0,T)$ for some $T\in(0,\infty]$ and with initial condition $\omega_0 \in L^\infty (\Omega)$.
Then $\omega(t,X_t^x)=\omega_0(x)$ for a.e.~$t\in(0,T)$ and a.e.~$x\in\Omega$ with $t_x>t$,  the velocity $u$ is continuous on $[0,T) \times\Omega$ (as well as on $[0,T] \times\Omega$ if $T<\infty$), and \eqref{1.333} holds for all $x\in\Omega$ and $t\in[0,t_x)$.
\end{cor}

{\it Remark.}
So even when $\partial\Omega$ is very irregular, vorticity might be created (at $\partial\Omega$) only if enough particle trajectories ``depart'' from the boundary into $\Omega$, so that $\Omega\setminus\{X_t^x\,|\,x\in\Omega\}$ has positive measure for some $t\in(0,T)$.

\medskip
\noindent {\bf Organization of the Paper and Acknowledgements}
\medskip

We prove Theorem \ref{T.1.1}(i) in Section \ref{S3a}, and then show in Section \ref{S2} how Theorem~\ref{T.1.1}(ii,iii) follows from it (the proof of Theorem \ref{T.1.1}(ii) also yields Corollary \ref{C.2.2}).  The proof of Lemma~\ref{T.3.1}, a crucial estimate used to obtain Theorem \ref{T.1.1}(i), appears in Section \ref{S3} (with a technical lemma used in it proved in Section \ref{S5}). We note that this proof becomes much simpler when the forward tangent vector $\bar\beta_\calT$ is itself Dini continuous (see the start of Section~\ref{S3}). The proof of Theorem~\ref{T.1.2}, which is related to that of Lemma \ref{T.3.1}, follows it in Section \ref{S4}.

We thank Claude Bardos, Camillo De Lellis, Peter Ebenfelt, Christophe Lacave, and Ming Xiao for helpful pointers to literature.
ZH acknowledges partial support by NSF grant DMS-1652284.
AZ was  supported in part by NSF grants DMS-1652284 and DMS-1900943.

\section{Proof of Theorem \ref{T.1.1}(i)} \lb{S3a}

Take any $x\in\Omega$ and let
\[
d(t):=1- | \calT(X_{t}^{x}) |
\]
be the distance of $\calT(X_{t}^{x})$ from $\partial\bbD$.  Then we have
\begin{equation*}
 d'(t)=- \frac{ \calT(X_{t}^{x}) }{| \mathcal{T}(X_{t}^{x})| } \cdot D\mathcal{T}(X_{t}^{x}) \frac{d}{dt} X_{t}^{x}
\end{equation*}
as long as $|\mathcal{T}(X_{t}^{x})|\in(0,1)$.
Since $D\mathcal{T}$ is of the form $\begin{pmatrix} a & b \\ -b & a \end{pmatrix}$ because $\calT$ is analytic,  we have $D\mathcal{T}D\mathcal{T}^T=(\det D\mathcal{T}) I_{2}$. The Biot-Savart law \eqref{1.111} for $\frac{d}{dt} X_{t}^{x}$ now shows that
\begin{align*}
 d'(t)= &
- \frac{ \det D\mathcal{T}(X_{t}^{x}) }{2\pi | \mathcal{T}(X_{t}^{x})| } \int_{\Omega} \left( \frac{ - \mathcal{T}(X_{t}^{x}) \cdot \mathcal{T}(y)^\perp}{| \mathcal{T}(X_{t}^{x}) - \mathcal{T}(y)|^2} + \frac{ \mathcal{T}(X_{t}^{x}) \cdot \mathcal{T}(y)^{*\perp}}{| \mathcal{T}(X_{t}^{x})- \mathcal{T}(y)^*|^2 } \right) \omega(t,y)\, dy \\
 =& \frac{ \det D\mathcal{T}(X_{t}^{x}) (1 - | \mathcal{T}(X_{t}^{x}) |^2)}{2\pi | \mathcal{T}(X_{t}^{x})| } \int_{\Omega} \frac{|\mathcal{T}(y)|^2 (1-|\mathcal{T}(y)|^2) \mathcal{T}(X_{t}^{x}) \cdot \mathcal{T}(y)^\perp }{| \mathcal{T}(X_{t}^{x}) - \mathcal{T}(y)|^2 \, | |\mathcal{T}(y)|^2 \mathcal{T}(X_{t}^{x})- \mathcal{T}(y)|^2 } \omega(t,y)\, dy.
\end{align*}
where $z^*:=z|z|^{-2}$ and $(a,b)^\perp:=(-b,a)$.  
After the change of variables  $z=\calT(y)$, we obtain
\[
 |d'(t)| \leq d(t)
\frac{ 2 \| \omega \|_{L^\infty}  }{\pi | \mathcal{T}(X_{t}^{x})| }  \det D\mathcal{T}(X_{t}^{x})\int_{\bbD} \frac{ (1- |z|) |\mathcal{T}(X_{t}^{x}) \cdot z^\perp | }{| \mathcal{T}(X_{t}^{x}) - z|^2 \, | |z|^2 \mathcal{T}(X_{t}^{x})- z|^2 } \det D\mathcal{T}^{-1}(z) \,dz.
\]

This estimate already appeared in \cite{LacZla}, but we will use the following crucial result to tightly bound its right-hand side for much more general domains.

\begin{lemma}\label{T.3.1}
Assume {\bf (H)} and that $\int_0^1 \frac {ds}{q_m(s)}=\infty$.
There is  $C<500$ and a  ($\calT$-dependent) constant $C_{\mathcal T}>0$ 
 such that if $|\xi|\in[\frac 12,1)$, then
\beq \lb{3.1}
\det D\mathcal{T}(\mathcal T^{-1}(\xi)) \int_\bbD \frac{(1-|z|)|\xi \cdot z^{\perp}|}{|\xi-z|^2 \, ||z|^2\xi-z|^2} \det D\mathcal{T}^{-1}(z)\,dz \leq  C\, Q_m(1-|\xi|) \left( \int_{1-|\xi|}^1\frac {ds}{sQ_m(s)} + C_{\mathcal T} \right),
\eeq		 
with $Q_m(s):=s^{-1}q_m(s)=\exp \left( \frac 2\pi \int_{s}^1 \frac{m(r)}r dr \right)$.
\end{lemma}
	
{\it Remark.}  
Note that $Q_m$ is non-increasing, and $\lim_{s\to 0} s^\alpha Q_m(s)=0$ for all $\alpha>0$ because $s^\alpha= \exp(\alpha\int_s^1\frac{dr}r)$.
\smallskip

Lemma \ref{T.3.1} with $\xi:=\calT(X_{t}^{x})$ now yields
\[
d'(t)\geq - C \| \omega \|_{L^\infty} q_m(d(t)) \left( \int_{d(t)}^1\frac {ds}{q_m(s)} + C_{\mathcal T} \right)
\] 
when $d(t)\in(0,\frac 12]$, 
with some $C<500$ and $C_\calT>0$.
Hence
\[
\frac d{dt} \ln \left(\int_{d(t)}^1\frac {ds}{q_m(s)} + C_{\mathcal T}  \right) \le C\| \omega \|_{L^\infty},
\]
and so
\[
\ln \int_{d(t)}^1\frac {ds}{q_m(s)} \le C\| \omega \|_{L^\infty} t + \ln  \left( \int_{\min\{d(0),1/2\}}^1\frac {ds}{q_m(s)} + C_{\mathcal T} \right)  
\]
for all $t\ge 0$.
Therefore 
\beq\lb{2.2}
d(t) \ge \rho_m  \left( C\| \omega \|_{L^\infty} t + \ln  \left( \int_{\min\{d(0),1/2\}}^1\frac {ds}{q_m(s)} + C_{\mathcal T} \right) \right).
\eeq
This is no less than   $\rho_m  \left( 500\| \omega \|_{L^\infty} t \right)$ for all large $t\ge 0$, uniformly in all $x$ with $|\calT(x)|\le R$ (for any $R<1$, except when $\omega\equiv 0$).  And if $M:=\int_0^1\frac {m(r)}r dr<\infty$, then $\rho_m(y)\ge \exp({-e^{y+2M/\pi}})$, so this is no less than $\exp({-e^{500 \| \omega \|_{L^\infty} t}})$ for all large $t\ge 0$, uniformly in all $x$ with $|\calT(x)|\le R$.

Hence, to conclude Theorem \ref{T.1.1}(i), it only remains to prove Lemma \ref{T.3.1}.
Since the proof is more involved, we do so in Section \ref{S3} below, after first showing how to obtain Theorem~\ref{T.1.1}(ii,iii) from Theorem \ref{T.1.1}(i).

\section{Proofs of Theorem \ref{T.1.1}(ii,iii) and Corollary \ref{C.2.2}} \lb{S2}

Theorem~\ref{T.1.1}(iii) follows immediately from Theorem \ref{T.1.1}(ii) and Proposition 3.2 in \cite{LacZla}, which shows that solutions from Theorem \ref{T.1.1}(ii) are unique as long as they remain constant near $\partial\Omega$  (constancy near the non-$C^{2,\gamma}$ portion of $\partial\Omega$ for some $\gamma>0$, where $u$ may be far from Lipschitz, is in fact sufficient).  It therefore suffices to prove Theorem \ref{T.1.1}(ii).

The first claim follows from the fact that the estimate \eqref{2.2} equally applies to the solutions of the time-reversed ODE $\frac d{ds} Y(s)=-u(t-s,Y(s))$ with $Y(0)\in\Omega$ (which of course satisfy $Y(s)=X_{t-s}^{Y(t)}$). 
The proof of the second  claim was obtained in \cite{Lacave-SIAM, LMW, LacZla} for some  sufficiently regular domains by looking at \eqref{1.1} as a (passive) transport equation with given $u$ and  $\omega_0$, and proving uniqueness of its solutions (using also that $t_x=\infty$ for all $x\in\Omega$).  This is because $\tilde\omega(t,X_t^x):=\omega_0(x)$ can be shown to be its weak solution in the sense of \eqref{1.222}.  The uniqueness proofs used the DiPerna-Lions theory, which required relevant extensions of $u$ and $\omega$ to $\bbR^2\setminus\Omega$ (the latter by 0).  This necessitated $\partial\Omega$ to be piecewise $C^{1,1}$, in addition to having $t_x=\infty$ for all $x\in\Omega$, so that the extension of $u$ is sufficiently regular for the DiPerna-Lions theory to be applicable.

We avoid this extension argument, and hence also extra regularity hypotheses on $\Omega$,
thanks to the following result concerning weak solutions to the  transport equation \eqref{1.1}.

\begin{lemma} \lb{L.2.1}
Let $\Omega\subseteq\bbR^d$ be open and $T\in(0,\infty]$.   Let $u\in L^\infty_{\rm loc}([0,T]\times\Omega)$ satisfy
\beq \lb{2.4}
\sup_{t\in[0,T]}\sup_{x,y\in K} \frac {|u(t,x)-u(t,y)|}{|x-y| \max \{1,-\ln|x-y|\} } <\infty
\eeq
for any compact $K\subseteq\Omega$, as well as \eqref{1.0b} on $(0,T)\times\Omega$.
If $\omega\in L^\infty_{\rm loc} ([0,T]\times\Omega)$ is a weak solution to \eqref{1.1} with initial condition $\omega_0 \in L^\infty_{\rm loc} (\Omega)$ and $X_t^x$ is from \eqref{1.333}, then we have $\omega(t,X_t^x)=\omega_0(x)$ for a.e.~$t\in(0,T)$ and a.e.~$x\in\Omega$ with $t_x>t$.
\end{lemma}

\begin{proof}
Let $\Omega_1\subseteq\Omega_2\subseteq\dots$ be smooth open bounded sets in $\bbR^2$ with $\bar\Omega_n\subseteq\Omega = \bigcup_{n\ge 1} \Omega_n$.  Since $\omega$ is also a weak solution to \eqref{1.1} on $\Omega_n$ and exit times $t_{x,n}$ of $X_t^x$ from $\Omega_n$ then satisfy $\lim_{n\to\infty} t_{x,n}=t_x$ for each $x\in\Omega$, it obviously suffices to prove that $\omega(t,X_t^x)=\omega_0(x)$ for a.e.~$t\in(0,T)$ and a.e.~$x\in\Omega_n$ such that $t_{x,n}>t$.  We can therefore assume that $\Omega$ is smooth and bounded, \eqref{2.4} holds with $K$ replaced by $\Omega$, and  $u,\omega,\omega_0$ are all bounded.
We can also assume without loss that $\omega\ge 0$ and $\omega_0\ge 0$, by adding a large constant to them. 

Extend the particle trajectories from \eqref{1.333} by $X_t^x:=\lim_{s\uparrow t_x}X_{s}^x\in\partial\Omega$ for all $t\ge t_x$, and let $\Omega_t:=\{X_t^x \,|\, x\in\Omega \,\,\&\,\, t_x>t\}$ for all $t\in[0,T)$ (these sets are open due to \eqref{2.4}).
 Then the lemma essentially follows from Theorem 2 in \cite{BonGus} but in order to apply it, we need to show that $\omega$ weakly satisfies some boundary conditions on $(0,T)\times\partial\Omega$ (even though these do not affect the result).  
 To this end we employ Theorem 3.1 and Remark 3.1 in \cite{Boyer}, which show that there is indeed some $\kappa\in L^\infty((0,T)\times\partial\Omega)$ such that
\[
\int_0^{T} \int_\Omega \omega \left(  \partial_t \varphi +  u \cdot \nabla \varphi \right) \,dx dt= -\int_\Omega \omega_0 \varphi(0, \cdot) \,dx + \int_0^T\int_{\partial\Omega}  (u\cdot n) \varphi \kappa \,d\sigma dt
\]
holds for all $\varphi \in C_0^\infty \left([0, T) \times \bar\Omega\right)$. 

Theorem 2 in \cite{BonGus} now shows that there is a positive measure $\eta$ on $\Omega$ such that
\beq \lb{2.1}
\int_{\Omega_t} \psi(y)\omega(t,y)\,dy = \int_\Omega \psi(X_t^x)\, d\eta(x)
\eeq
for almost all $t\in(0,T)$ and all $\psi\in C_0^\infty(\Omega_t)$.  (In fact, the measure in \cite{BonGus} is supported on the set of all maximal solutions to the ODE $\frac d{dt}Y(t)=u(t,Y(t))$ on $(0,T)$, 
and the relevant formula holds for all $\psi\in C_0^\infty(\bbR^d)$.  But this  becomes \eqref{2.1} when restricted to the $\psi$ above, with $\eta$ the restriction of the measure from \cite{BonGus} to the set of solutions $\{\{X_t^x\}_{t\in(0,T)}\,|\, x\in\Omega\}$.  This is because uniqueness of solutions for the ODE shows that the other solutions have $Y(t)\notin\Omega_t$ for any $t\in(0,T)$.)  By taking $t\to 0$ in \eqref{2.1}, we obtain
\[
\int_{\Omega} \psi(y)\omega_0(y)\,dy = \int_\Omega \psi(x)\, d\eta(x)
\]
 for  any $\psi\in C_0^\infty(\Omega)$, so $d\eta(x)=\omega_0(x)dx$. 
Letting $\psi$ in \eqref{2.1} be approximate delta functions near all $y\in\Omega_t$ then shows that for almost all $t\in(0,T)$ we have $\omega(t,X_t^x)=\omega_0(x)$ whenever $x$ and $X_t^x$ are Lebesgue points of $\omega_0$ and $\omega(t,\cdot)$, respectively.  This finishes the proof.
\end{proof}

Since  $t_x=\infty$ for all $x\in\Omega$, Lemma \ref{L.2.1} with 
$T=\infty$ now 
proves the second claim in Theorem \ref{T.1.1}(ii).  As in \cite{LacZla}, uniform boundedness of $u$ on any compact subset of $\Omega$ then yields $\omega\in C([0,\infty);L^1(\Omega))$, and continuity of $u$ on $[0,\infty)\times \Omega$ follows from this and the Biot-Savart law.  Then also \eqref{1.333} holds pointwise, finishing the proof of Theorem \ref{T.1.1}(ii).

This argument actually applies on general open bounded $\Omega\subseteq\bbR^2$, without needing $t_x=\infty$ for all $x\in\Omega$.  This is because boundedness of $\omega$ implies $u\in L^\infty ((0,T)\times K)$ for any compact $K\subseteq\Omega$  as well as \eqref{2.4} (for solutions on a time interval $(0,T)$ with $T<\infty$), and these three facts then again yield $\omega\in C([0,T];L^1(\Omega))$ (with $\omega(0,\cdot):=\omega_0$ and $\omega(T,\cdot)$ defined by continuity).  This yields Corollary \ref{C.2.2}.

\section{Proof of Lemma \ref{T.3.1}} \lb{S3}

We can assume that $\tilde\beta_\calT(0)=0$, which is achieved by subtracting $\tilde\beta_\calT(0)$ from $\tilde\beta_\calT$ and adding it to $\beta_\calT$.
Since $\calT$ is analytic, we have  $\det D\calT(z)=|\calT'(z)|^2$, where $\calT'$ is the complex derivative when $\calT$ is considered as a function on $\bbC$.  The same is true for its inverse $\calS$, and we also have $\calS'(z)=\calT'(\calS(z))^{-1}$.
Since $\Omega$ is regulated, Theorem~3.15 in \cite{Pomm}  shows that
\beq \lb{3.2y'}
\mathcal{S}'(z) = |\mathcal{S}'(0)| \, \exp \left(\frac{i}{2\pi} \int_0^{2\pi} \frac{e^{i\tht}+z} {e^{i\tht}-z} \left( \bar \beta_\calT(\tht) -\tht-\frac\pi 2 \right) d\tht \right)
\eeq
for all $z\in\bbD$, and from $\int_0^{2\pi} \frac{e^{i\tht}+z} {e^{i\tht}-z}\,d\tht = 2\pi \in\bbR$ and ${\rm Im}\, \frac{e^{i\tht}+z} {e^{i\tht}-z} = 2 {\rm Im}\, \frac{z} {e^{i\tht}-z}$ we get
\beq \lb{3.2'}
\det D\calS (z)  = \det D\calS (0) \, \exp \left( -\frac{2}{\pi} \int_0^{2\pi}  {\rm Im}\, \frac{z} {e^{i\tht}-z} \left( \bar \beta_\calT(\tht) -\tht \right) d\tht \right)
\eeq
(with $\bar \beta_\calT(\tht) -\tht$ being $2\pi$-periodic).

We note that if $\bar\beta_\calT$ is itself Dini continuous (so we can have $\tilde\beta_\calT=\bar\beta_\calT$ and $\int_0^1\frac {m(r)}r\,dr<\infty$),
then the integral in \eqref{3.2'} is uniformly bounded by some $m$-dependent constant.  Indeed, letting $\tht_z:=\arg z$, this follows from oddness of  ${\rm Im}\, \frac{z} {e^{i(\tht-\tht_z)}-z}$ in $\tht$, together with the bound $| \frac{z} {e^{i\tht}-z}|\le \frac \pi {2|\tht-\tht_z|}$ and therefore
\[
\left | \frac{z} {e^{i\tht}-z} \left( \bar \beta_\calT(\tht) - \bar \beta_\calT(\tht_z) \right) \right| \le \frac\pi 2 \,\frac {m(|\tht-\tht_z|)}{|\tht-\tht_z|}.
\]  
One can also easily show that $\int_\bbD \frac{(1-|z|)|\xi \cdot z^{\perp}|}{|\xi-z|^2 \, ||z|^2\xi-z|^2}dz\le C(\ln(1-|\xi|) +1)$ for some $C>0$ when $|\xi|\in[\frac 12,1)$, using some simple estimates appearing right after the proof of Lemma \ref{CorrOfFolding} below.
So \eqref{3.1}  with the right-hand side $C_m(\ln(1-|\xi|) +1)$ follows immediately in this case.  The rest of this section (and Section \ref{S5}) proves \eqref{3.1} in the general case.

We will now split the exponential in \eqref{3.2'} into the parts corresponding to $\beta_\calT$ and $\tilde\beta_\calT$.  Let $\kappa:=\frac 1{2\pi}(\tilde\beta_\calT(2\pi)-\tilde\beta_\calT(0))$, so that $\tilde\beta_\calT(\tht)-\kappa\tht$ and $\beta_\calT(\tht)-(1-\kappa)\tht$ are both $2\pi$-periodic (note that we also have $\kappa\in[-\frac{m(2\pi)}{2\pi},\min\{1,\frac{m(2\pi)}{2\pi}\}]$ because $\beta_\calT$ is non-decreasing).  Integration by parts then shows that
\[
\int_0^{2\pi} \frac{z} {e^{i\tht}-z} \left( \beta_\calT(\tht) - (1-\kappa)\tht \right) d\tht = i \int_0^{2\pi} \ln(1-ze^{-i\tht})\, d\left( \beta_\calT(\tht) -(1-\kappa)\tht \right),
\]
so from $\int_0^{2\pi} \ln(1-ze^{-i\tht})  d\tht=\ln 1 =0$ we  obtain
\beq\lb{3.3'}
\int_0^{2\pi}  {\rm Im}\, \frac{z} {e^{i\tht}-z} \left( \beta_\calT(\tht) - (1-\kappa)\tht \right) d\tht = \int_0^{2\pi} \ln|e^{i\tht}-z| \,d \beta_\calT(\tht).
\eeq

In order to simplify notation, let $\beta$ be the positive measure with distribution function $\beta_\calT$, and define the function $\tbet(\tht):=\tilde\beta_\calT(\tht)-\kappa\tht$.  Then $\tilde\beta$ has modulus of continuity $\tilde m(r):=m(r)+|\kappa|r$, and we have $\tilde m(r) \le m(r)+\frac {m(2\pi)}{2\pi} r\le 3 m(r)$ for $r\in[0,2\pi]$.  This is because any modulus satisfies $m(a 2^{-n})\ge 2^{-n}m(a)$  for any $a\in[0,2\pi]$ and $n\in\bbN$ (by induction), and thus $m(b)\ge \frac b{2a}m(a)$ whenever $0\le b\le a\le 2\pi$
 since $m$ is non-decreasing.
  We also let 
\[
|\beta|:=\beta((0,2\pi])=\beta_\calT(2\pi)-\beta_\calT(0)=2\pi(1-\kappa)\in[0,2\pi+m(2\pi)].
\]
	Next, for any $z \in \mathbb{D}$, bounded measurable $A \subset \mathbb{R}$, and $\tht^*\in\bbR$, let 
		\begin{align*}
		\mathcal{I}(z, A) &:= \frac{2}{\pi} \int_{A} \ln |e^{i\tht}-z| \,d\beta(\tht ),
		\\ \mathcal{J}(z, A,\tht^*) &:= \frac{2}{\pi} \int_{A}  {\rm Im}\, \frac{z} {e^{i\tht}-z} \,(\tilde \beta (\tht)-\tilde\beta(\tht^*)) \,d\tht,
		\end{align*}
as well as
		\begin{align*}
		\mathcal{I}(z) & := \mathcal{I}(z, (0,2\pi]), 
		\\ \mathcal{J}(z) & := \mathcal{J}(z, (0,2\pi],\tht^*) 
		\end{align*}
(with the latter independent of $\tht^*$ due to $ \int_{0}^{2\pi}  {\rm Im}\, \frac{z} {e^{i\tht}-z} d\tht=0$).
Then \eqref{3.2'} and \eqref{3.3'} yield
\[
	\det D\calS (z)  = \det D\calS (0) \, e^{-\mathcal{I}(z)-\mathcal{J}(z)}
\]
	and
\beq\lb{3.2x}
	\det D\calT (\calS(z))  = \det D\calS (0)^{-1} \, e^{\mathcal{I}(z)+\mathcal{J}(z)}
\eeq
(recall that $\tilde\beta(0)=0$).
In view of this,  \eqref{3.1} becomes
\beq \lb{3.2}
\int_\bbD \frac{(1-|z|)|\xi \cdot z^{\perp}|}{|\xi-z|^2 \, ||z|^2\xi-z|^2} e^{\mathcal{I}(\xi)-\mathcal{I}(z)} e^{\mathcal{J}(\xi)-\mathcal{J}(z)} \, dz \leq C\, Q_m(1-|\xi|)  \left( \int_{1-|\xi|}^1\frac {ds}{sQ_m(s)} + C_\calT \right).
\eeq		 

%
%


To prove this, we  need  the following lemma, whose proof we postpone to Section \ref{S5}.



\begin{lemma} \label{FoldingLemma}
Let $\beta$ be a (positive) measure on $\bbR$ and let $I := [\tht^*-2\delta, \tht^*+2\delta]$ for some $\tht^*\in\bbR$ and $\delta\in(0,\frac\pi 2]$.
Let $H \subset 
\bbD$ be an open region such that if $re^{i(\tht^*+\phi)} \in H$ for some  $r\in(0,1)$ and $|\phi|\le\pi$, then $re^{i(\tht^*+\phi')} \in H$ whenever $|\phi'|\le|\phi|$ (i.e., $H$ is symmetric and angularly convex with respect to  the line connecting 0 and $e^{i\tht^*}$).
 If $\alpha\ge 1$, then 
	\begin{equation}\label{SecondInequality}
	\int_{H} f(z) \left[ g(z) + \frac{1}{\beta(I)} \int_{I} h(|e^{i\tht}-z|) d\beta(\tht) \right]^{\alpha} dz \leq  \int_{H} f(z) \left[ g(z) + h(|e^{i\tht^*}-z|) \right]^{\alpha}dz
	\end{equation}
	holds for any non-increasing $h:(0,\infty)\to[0,\infty)$ and non-negative $f,g\in L^1(H)$ such that $f(re^{i(\tht^*+\phi')}) \geq f(re^{i(\tht^*+\phi)})$ and $g(re^{i(\tht^*+\phi')}) \geq g(re^{i(\tht^*+\phi)})$ whenever $r\in(0,1)$ and  $|\phi'|\le |\phi|$.
\end{lemma} 

{\it Remark.}   The right-hand side of \eqref{SecondInequality} is just the left-hand side for the Dirac measure at $\tht^*$ with mass $\beta(I)$.  That is, concentrating all the mass of $\beta$ on $I$ into $\tht^*$ cannot decrease the value of the integral in \eqref{SecondInequality}.
\smallskip


Next, we claim that there is $\delta>0$ 
such that $\beta([\tht -2\delta,\tht +2\delta]) \leq \frac{4}{3}\pi$ for all $\tht  \in \bbR$ (any number from $(\pi,\frac 32\pi)$ would work in place of $\frac 43\pi$ here).  Let $\delta'>0$ be such that any interval of length $4\delta'$ contains at most one $\tht$ with $\beta(\{\tht \}) \geq \frac{\pi}{9}$ (there are only finitely many such $\tht $ in $(0,2\pi]$). 
Then for each $\tht  \in [0,2\pi]$, find $\delta_\tht \in(0,\delta']$ such that $\beta([\tht -2\delta_\tht ,\tht +2\delta_\tht ]) \leq \beta(\{\tht \})+ \frac{\pi}{9}$.  
Since  $\{ (\tht -2\delta_\tht , \tht +2\delta_\tht ) \,|\,  \tht  \in [-\pi,3\pi]   \}$ is an open cover of $[-\pi,3\pi]$, there is a finite  sub-cover $\{ (\tht _k-2\delta_{\tht _k}, \tht _k+2\delta_{\tht _k}) \,|\, k=1,\dots, N \}$.  If we let $\delta:= \min\{ \delta_{\tht _k} \,|\, k=1,\dots, N \}>0$, then indeed
$\beta([\tht -2\delta,\tht +2\delta]) \le (\pi + \frac{\pi}{9})+(\frac\pi 9+\frac\pi 9) = \frac{4}{3} \pi$ for all $\tht \in[0,2\pi]$ (and so for all $\tht \in\bbR$).  This is because $[\tht -2\delta,\tht +2\delta]\subseteq [\tht _k-2\delta_{\tht _k}, \tht _k+2\delta_{\tht _k}]\cup [\tht _j-2\delta_{\tht _j}, \tht _j+2\delta_{\tht _j}]$ for some $k,j$ such that $|\tht _k-\tht _j|\le 4\delta'$, and hence at most one of $\beta(\{\tht _k\})$ and $\beta(\{\tht _j\})$ is greater than $\frac {\pi}9$ (unless $k=j$), while obviously each is at most $\pi$.

Moreover, let us 
decrease this constant so that $\delta\in(0,\frac {\ln 2}{10^3(1+m(2\pi))}]$ and $m(2\delta)\le \frac {\ln 2}{300}$.
With this ($\calT$-dependent) $\delta$, we can now prove the following estimates (recall \eqref{3.2}).

\begin{lemma} \label{CorrOfFolding}
		Let $\beta,\tbet,m$ and $\delta$ be as above.  
	 	There are $C_{|\beta|,\delta}$ and $C_{m}$ 
		(depending only on $|\beta|,\delta$ and on $m$, respectively, so only on $\calT$) 
		such that for any $\xi\in\bbD$ we have 
		\begin{equation} \label{EqnInCor}
			\int_{\bbD} z^{-1}(1-|z|)^{5/6} e^{\mathcal{I}(\xi)-\mathcal{I}(z)}  dz \leq C_{|\beta|,\delta}, 
		\end{equation}
		and for all $z,\xi\in\bbD$ also
		\begin{equation} \label{3.5}
			e^{\mathcal{J}(\xi)-\mathcal{J}(z)} \leq 
			C_{m} \frac {Q_m(\min\{1-|\xi|,|\xi-z|\})}{Q_m(|\xi-z|)}\, \frac {Q_m(\min\{1-|z|,|\xi-z|\})}{Q_m(|\xi-z|)} . 
		\end{equation}
Moreover, if 
		 $|\xi-z|\le 4\delta$,  then for $\tht_\xi:=\arg \xi$
		 and $I:=[\tht_\xi-2\delta,\tht_\xi+2\delta]$ we have
	\begin{equation} \label{3.6}
			e^{\mathcal{J}(\xi,I,\tht_\xi)-\mathcal{J}(z,I,\tht_\xi)} \leq  2 \,\frac {Q_m(\min\{1-|\xi|,|\xi-z|\})}{Q_m(|\xi-z|)} \, \frac {Q_m(\min\{1-|z|,|\xi-z|\})}{Q_m(|\xi-z|)} .
		\end{equation}	
\end{lemma}



\begin{proof}
Let us start with \eqref{3.5}.   Let $\tht_\xi:=\arg \xi$ and $\tht_z:=\arg z$, as well as
\[
A:= \left \{ \tht\in(0,2\pi] \,\Big|\, \min \{ d (\tht,\tht_\xi), d (\tht,\tht_z) \} \ge \frac 12|\xi-z| \right \},
\] 
where $d$ is the distance in $[0,2\pi]$ with 0 and $2\pi$ identified.
Then from
\[
\left| \frac{\xi} {e^{i\tht}-\xi}  - \frac{z} {e^{i\tht}-z} \right| =  \left| \frac{e^{i\tht}(\xi-z)} {(e^{i\tht}-\xi)(e^{i\tht}-z)} \right| 
\le  \pi^2 \frac{|\xi-z|} {d(\tht,\tht_\xi)\,d(\tht,\tht_z)}
\]
we obtain with $a:=\frac 12 d(\tht_\xi,\tht_z)$ and $b:=\min\{\frac 12 |\xi-z|,a\}\le a$,
\[
|\calJ(\xi,A,\pi)-\calJ(z,A,\pi)| \le 4\pi |\xi-z| \tilde m(\pi) \left( \int_{|\xi-z|/2}^\pi \frac {dr}{r(r+2a)} + \int_{b}^a \frac {dr}{r(2a-r)}  \right).
\]
That is, 
\[
|\calJ(\xi,A,\pi)-\calJ(z,A,\pi)|\le 4\pi \tilde m(\pi) \left( 2+ \frac ba \ln \frac ab \right) \le 10 \pi \tilde m(\pi) \le  C_{m}.
\]

On the complement $A^c:=(0,2\pi]\setminus A$ we can estimate the two $\calJ$ terms individually.  To conclude  \eqref{3.5}, it now suffices to show 
\beq \lb{3.5z}
|\calJ(z,A^c,\pi)| \le C_{m} + \ln \frac {Q_m(\min\{1-|z|,|\xi-z|\})}{Q_m(|\xi-z|)}
\eeq
because an analogous estimate then follows for $\calJ(\xi,A^c,\pi)$ as well.  First note that if we let $A':=\{\tht\in A^c\,|\, d(\tht,\tht_z)> \frac 12 |\xi-z|\}$, then
\[
|\calJ(z,A',\pi)| \le 2\tilde m(\pi) \int_{d(A',\tht_z)}^{3 d(A',\tht_z)} \frac {dr}r \le C_m.
\]
With  $A'':=\{\tht\in A^c \,|\, d(\tht,\tht_z)\le \frac 12 \min\{1-|z|,|\xi-z|\} \}$ we also have
\[
|\calJ(z,A'',\pi)| \le \frac 2\pi \tilde m(\pi) \le C_m
\]
due to $|e^{i\tht}-z|\ge 1-|z|$, which proves \eqref{3.5z} when $|\xi-z|\le 1-|z|$.  If instead $|\xi-z|> 1-|z|$, then we also use  oddness of ${\rm Im}\, \frac{z} {e^{i(r+\tht_z)}-z}$ in $r$ 
and $|e^{i\tht}-z|\ge \sin|\tht-\tht_z|$ to estimate
\[
|\calJ(z,A^c\setminus(A'\cup A''),\pi)|
\le \frac 2\pi \int_{(1-|z|)/2}^{|\xi-z|/2} \frac{\tilde m(2r)}{\sin r} dr \le C_m + \ln \frac {Q_m(1-|z|)}{Q_m(|\xi-z|)},
\]
with the last inequality due to
\[
\int_{a/2}^{b/2} \frac{\tilde m(2r)}{\sin r} dr \le \int_{a}^{b} \frac{\tilde m(s)}{\sin s} ds \le \int_a^b \frac {m(s)}s ds + \int_a^b \left( 10\tilde m(s) + |\kappa| \right)ds \le \int_a^b \frac {m(s)}s ds + C_m
\]
for $0\le a\le b\le 2$ (because $\sup_{s\in[0,2]}(\frac1{\sin s}-\frac 1s)\le 10$).  Hence \eqref{3.5z} follows, proving \eqref{3.5}.

To obtain \eqref{3.6}, we repeat this argument with some minor adjustments.  For
\[
A:= \left \{ \tht\in I \,\Big|\, \min \{ d (\tht,\tht_\xi), d (\tht,\tht_z) \} \ge \frac 12|\xi-z| \right \},
\] 
we obtain the bound
\[
|\calJ(\xi,A,\tht_\xi)-\calJ(z,A,\tht_\xi)|\le 4\pi \tilde m(2\delta) \left( 2+ \frac ba \ln \frac ab \right) \le 10 \pi \tilde m(2\delta) \le  30 \pi  m(2\delta) \le \frac { \ln 2} 3.
\]
(recall that $\tilde m(s)\le 3m(s)$). Hence it suffices to show \eqref{3.5z} with $A^c:=I\setminus A$, and with $\tht_\xi$ and  $\frac { \ln 2} 3$ in place of $\pi$ and $C_m$.
As above, we now obtain
\[
|\calJ(z,A',\tht_\xi)| \le 2\tilde m(2\delta) \int_{d(A',\tht_z)}^{3 d(A',\tht_z)} \frac {dr}r \le 4\ln 3\, m(2\delta) \le \frac {\ln 2}9
\]
and
\[
|\calJ(z,A'',\tht_\xi)| \le \frac 2\pi \tilde m(2\delta) \le   \frac 4\pi m(2\delta) \le \frac {\ln 2}9.
\]
Finally, if $|\xi-z|> 1-|z|$, then we also obtain
\[
|\calJ(z,A^c\setminus(A'\cup A''),\tht_\xi)|
\le \frac 2\pi \int_{(1-|z|)/2}^{|\xi-z|/2} \frac{\tilde m(2r)}{\sin r} dr \le \frac {\ln 2}9 + \ln \frac {Q_m(1-|z|)}{Q_m(|\xi-z|)}
\]
because $\int_a^b \left( 10\tilde m(s) + |\kappa| \right)ds\le 4\delta( 21m(2\pi))\le  \frac {\ln 2}9$ when $0\le a\le b\le 4\delta$.

Now we prove \eqref{EqnInCor}.  We obviously have 
\beq \lb{3.6b}
\max\{\mathcal{I}(\xi),|\calI(z)|\} \le \frac {2\ln 2}\pi |\beta|
\eeq 
for all $\xi\in\bbD$ and all $z\in B(0,\frac 12)$, so  it suffices to prove 
\begin{equation} \label{EqnInCor'}
			\int_{\bbD} (1-|z|)^{5/6} e^{-\mathcal{I}(z)}  dz \leq C_{|\beta|,\delta}. 
\end{equation}			
The integrand  is clearly bounded above by $(\frac\delta 2)^{-2|\beta|/\pi}$ on  $B(0,1-\frac \delta 2)$.  Since $\bbD\setminus B(0,1-\frac \delta 2)$ can be covered by $O(\frac 1\delta)$ disks with centers on $\partial\bbD$ and  radii $\delta$, 
it suffices to prove \eqref{EqnInCor'} with $H:=B(e^{i\tht^*},\delta)\cap\bbD$
in place of $\bbD$,
for any $\tht^*\in\bbR$.

	Let $I: = [\tht^*-2\delta, \tht^* + 2\delta]$ and $\alpha := \frac {2\beta(I)}\pi\in[0,\frac 83]$. 
	Since $\calI (z,(0,2\pi]\setminus \bigcup_{k\in\bbZ} (I + 2k\pi))$ is bounded below by $\frac {2|\beta|}\pi \ln \frac\delta 2$  for all $z \in H$,
	it in fact suffices to prove
		\begin{equation} \label{3.7}
			\int_{H} (1-|z|)^{5/6} e^{-\mathcal{I}(z,I)}dz \leq C.
		\end{equation}	
If $ \alpha \in [0,1]$, then from $1-|z| \leq |e^{i\tht}-z|$ for all $(z,\tht)\in\bbD\times \mathbb{R}$ we indeed have 
\[
\int_{H} (1-|z|)^{5/6} e^{-\mathcal{I}(z,I)}dz = \int_{H} (1-|z|)^{-\alpha+5/6} \exp\left(\frac{2}{\pi}\int_{I} \ln \frac{1-|z|}{|e^{i\tht}-z|} d\beta(\tht) \right)dz \le \int_\bbD (1-|z|)^{-1/6}dz,
\]
as needed.
If  $\alpha \in [1,\frac{8}{3}]$, then we instead use Jensen's inequality and Lemma \ref{FoldingLemma} with $f(z)=(1-|z|)^{5/6}$, $g(z) = 0$, and $h(s)=\frac 1s$ to obtain
\begin{eqnarray*}
   	\int_{H} (1-|z|)^{5/6} e^{-\mathcal{I}(z,I)}dz &\leq&
	\int_{H} (1-|z|)^{5/6} \exp\left[ \alpha \ln \left( \frac 1{\beta(I)} \int_{I} \frac 1{ |e^{i\tht}-z|} d\beta(\tht) \right) \right] dz \\
	&=& \int_{H} (1-|z|)^{5/6} \left(\frac{1}{\beta(I)} \int_{I} \frac{1}{|e^{i\tht}-z|} d\beta(\tht) \right)^{\alpha}     dz \\
 	&\leq& \int_{H} (1-|z|)^{5/6} {|e^{i\tht^*}-z|^{-\alpha}} dz \\
 	&\leq& \int_{H} |e^{i\tht^*}-z|^{-\alpha+5/6} dz \\
 	&\leq& 12\pi. 
\end{eqnarray*}
This proves \eqref{3.7} 
and hence also \eqref{EqnInCor}.
\end{proof}

Now we are ready to prove Lemma \ref{T.3.1}

\begin{proof}[Proof of Lemma \ref{T.3.1}]
For the sake of simplicity, we will prove the result with $C<10^5$ first, and then indicate the changes required to obtain $C<500$. 
Consider the  ($\calT$-dependent) $\delta$ from above.
Recall that we only need to prove \eqref{3.2}, and
note that $\xi \cdot z^{\perp}=(\xi-z) \cdot z^{\perp}$ implies
\beq\lb{3.8}
\frac{|\xi \cdot z^{\perp}|}{|\xi-z|^2 \, ||z|^2\xi-z|^2} 
\le  \frac{1}{|\xi-z|  \, |z|  \, | |z|\xi- \frac z{|z|}|^2} = \frac{1}{|\xi-z|  \, |z|^3  \, |\xi- \frac z{|z|^{2}}|^2}.
\eeq
Together with \eqref{3.6b} and \eqref{3.5} this yields $C_{m}$ such that for any $\xi\in\bbD\setminus B(0,\frac 12)$,
\[
 \int_{B(0,\frac 14)} \frac{(1-|z|)|\xi \cdot z^{\perp}|}{|\xi-z|^2||z|^2\xi-z|^2} e^{\mathcal{I}(\xi)-\mathcal{I}(z)} e^{\mathcal{J}(\xi)-\mathcal{J}(z)} dz
 \leq C_{m} Q_m(1-|\xi|)
\]
because the last fraction in \eqref{3.5} is bounded above by $\exp \left( \frac 2\pi \int_{3/4}^{5/4} \frac{m(r)}r dr \right)$ when $z\in B(0,\frac 14)$
(the dependence of the constant on $|\beta|$ need not be indicated because $0\le|\beta|\le 2\pi+m(2\pi)$). 

If now $|\xi|\in[\frac 12,1)$ and $z\in B(\xi,\frac{1-|\xi|}2)$, then $\mathcal{I}(\xi)-\mathcal{I}(z)\leq \frac {2|\beta|}\pi$ due to $|e^{i\tht}-\xi||e^{i\tht}-z|^{-1} \le 2$ for all $\tht \in \mathbb{R}$.   Hence using  $|\xi- \frac z{|z|^{2}}|\ge 1-|\xi|\ge\frac{1-|z|}2$ in \eqref{3.8} (because $\frac z{|z|^{2}}\notin\bbD$) and $|\xi-z|\le\min\{1-|\xi|,1-|z|\}$ in \eqref{3.5} yields
\[
 \int_{B(\xi,\frac {1-|\xi|}2)} \frac{(1-|z|)|\xi \cdot z^{\perp}|}{|\xi-z|^2||z|^2\xi-z|^2} e^{\mathcal{I}(\xi)-\mathcal{I}(z)}e^{\mathcal{J}(\xi)-\mathcal{J}(z)} dz
\le C_m  \int_{B(\xi,\frac {1-|\xi|}2)}  \frac 1{|\xi-z|(1-|\xi|)} dz\le C_m\pi.
\]
For all other $z\in \bbD\setminus B(0,\frac 14)$, we can bound the right-hand side of \eqref{3.8} above by $\frac{64}{|\xi-z|^3}$, using that $|\frac z{|z|^{2}}|-1\ge 1-|z|$ implies $|\xi-\frac z{|z|^{2}}|\ge|\xi-z|$. This, \eqref{3.5}, \eqref{EqnInCor}, and the bound $Q_m(1-|z|) \le C_m (1-|z|)^{1/6}$ (see the remark after Lemma \ref{T.3.1}) now yield
\[
 \int_{\bbD\setminus (B(\xi,\delta^3)\cup B(0,\frac 14))} \frac{(1-|z|)|\xi \cdot z^{\perp}|}{|\xi-z|^2||z|^2\xi-z|^2} e^{\mathcal{I}(\xi)-\mathcal{I}(z)} e^{\mathcal{J}(\xi)-\mathcal{J}(z)} dz
 \leq   C_{m,\delta} \,Q_m(1-|\xi|).
\]
To obtain \eqref{3.2}, it therefore suffices to prove
\[
 \int_{H_\xi} \frac{1-|z|}{|\xi-z|^3} e^{\mathcal{I}(\xi)-\mathcal{I}(z)}e^{\mathcal{J}(\xi)-\mathcal{J}(z)} dz \le C\,Q_m(1-|\xi|) \left( \int_{1-|\xi|}^1\frac {ds}{sQ_m(s)} + 1 \right)
\]
when $|\xi|\in[1-2\delta^3,1)$, with $H_\xi:=[ B(\xi,\delta^3)\setminus B(\xi,\frac{1-|\xi|}2)]\cap \bbD$ and a universal $C<10^5(1-3\delta^3)^3$.  Since $(1-3\delta^3)^3\ge (1-\frac 3{10^9})^3 >  1-\frac 1{10^8}$, it suffices to obtain $C\le 10^5-1$ here

Let $\tht_\xi:=\arg \xi$, and again let $I:=[\tht_\xi-2\delta, \tht_\xi + 2\delta]$ as well as $\alpha := \frac {2\beta(I)}\pi\in[0,\frac 83]$.  Then $|e^{i\tht}-\xi| \ge\delta $ for all  $\tht  \notin \bigcup_{k\in\bbZ} (I + 2k\pi)$, hence for all such $\tht $ and all  $z\in B(\xi,\delta^3)$ we have $\frac{|e^{i\tht}-\xi|}{|e^{i\tht}-z|}\le \frac1{1-\delta^2} \le  1+\frac\pi{2|\beta|}$ (the last inequality follows from $\delta^2\le\frac\pi{\pi+2|\beta|}$, which is due to $\frac\pi{\pi+2|\beta|}\ge \frac\pi{5\pi+2m(2\pi)} \ge  \frac {\ln 2}{10^3(1+m(2\pi))}\ge\delta$).  This yields for all $z\in B(\xi,\delta^3)$,
\beq\lb{3.9x}
\mathcal{I}(\xi)-\mathcal{I}(z) = \frac{2}{\pi} \int_{(0,2\pi]} \ln\frac{|e^{i\tht}-\xi|}{|e^{i\tht}-z|}d\beta(\tht)  \leq 1+ \frac{2}{\pi} \int_{I} \ln\frac{|e^{i\tht}-\xi|}{|e^{i\tht}-z|}d\beta(\tht). 
\eeq
Similarly, for the same $z$ and $\tht$ we have $| \frac{\xi} {e^{i\tht}-\xi}  - \frac{z} {e^{i\tht}-z} | =   \frac{|\xi-z|} {|e^{i\tht}-\xi||e^{i\tht}-z|} \le \frac\delta{1-\delta^2}  \le \frac 1{4\tilde m(2\pi)}$, so
\beq \lb{3.9y}
\mathcal{J}(\xi)-\mathcal{J}(z) = \mathcal{J}(\xi,(0,2\pi],\tht_\xi)-\mathcal{J}(z,(0,2\pi],\tht_\xi)  
\le 1+ \mathcal{J}(\xi,I,\tht_\xi)-\mathcal{J}(z,I,\tht_\xi).
\eeq
Using \eqref{3.6}, combined with $Q_m(\frac 12(1-|\xi|))Q_m(1-|\xi|)^{-1} \le e^{2m(2\delta^3)/\pi}\le e^{1/100\pi}$ (recall that $|\xi-z|\ge \frac {1-|\xi|}2$) and $Q_m(a)Q_m(b)^{-1}\le \exp(\frac 16\int_a^b \frac 1r dr) = b^{1/6}a^{-1/6}$ for $0<a\le b\le \delta^3$ (because $m(\delta^3)\le m(2\delta)\le \frac\pi {12}$),  we thus obtain
\beq \lb{3.9z}
e^{\mathcal{J}(\xi)-\mathcal{J}(z)}  \le 2\cdot3^{1/6}e^{1+ 1/100\pi} \, \frac {Q_m(1-|\xi|)}{Q_m(|\xi-z|)} \, \frac {|\xi-z|^{1/6}}{(1-|z|)^{1/6}} ,
\eeq
where we also used $1-|z|\le 3|\xi-z|$ for all  $z\in \bbD\setminus B(\xi,\frac{1-|\xi|}2)$.
Estimates \eqref{3.9x} and \eqref{3.9z}, together with $2\cdot 3^{1/6}e^{1+ 1/100\pi}\le 3e$ and 
\beq \lb{3.9w}
\int_{\frac 12(1-|\xi|)}^{1-|\xi|} \frac {ds}{sQ_m(s)} \le \frac {\ln 2}{Q_m(1-|\xi|)} \le  \int_{1-|\xi|}^{2(1-|\xi|)} \frac {ds}{sQ_m(s)} \le \int_{1-|\xi|}^{1} \frac {ds}{sQ_m(s)},
\eeq
now show that it suffices to prove
\beq\lb{3.9}
 \int_{H_\xi}  \frac{(1-|z|)^{5/6}}{|\xi-z|^{17/6}} \exp \left( \frac{2}{\pi} \int_{I} \ln\frac{|e^{i\tht}-\xi|}{|e^{i\tht}-z|}d\beta(\tht) \right)  \frac {dz}{Q_m(|\xi-z|)}  \leq C \left( \int_{\frac 12(1-|\xi|)}^{1}\frac {ds}{sQ_m(s)} + 1 \right)
\eeq
whenever $|\xi|\in[1-2\delta^3,1)$, with some universal  $C\le \frac {10^5 -1}{6e^2}$.  

%


Consider now the case $\alpha\in[0,1]$.  
We have $1-|z|\le |e^{i\tht}-z|$ for all $(z,\tht )\in\bbD\times\bbR$, and $1-|z|\le 3|\xi-z|$ for all  $z\in H_\xi$.  This and the triangle inequality yield 
\beq\lb{3.10}
\frac{|e^{i\tht}-\xi|}{|e^{i\tht}-z|} \le 
 \frac{|\xi-z|}{|e^{i\tht}-z|} +1  \le 4 \frac{|\xi-z|}{1-|z|}
\eeq
for all $(z,\tht )\in H_\xi\times I$.  Therefore the left-hand side of \eqref{3.9} is  bounded above by
\begin{align*}
  4^\alpha  \int_{H_\xi}  \frac{(1-|z|)^{-\alpha+5/6}}{|\xi-z|^{-\alpha+17/6}}  \frac {dz}{Q_m(|\xi-z|)}
 & \le  4^\alpha 3^{1-\alpha} \int_{H_\xi} \frac {(1-|z|)^{-1/6}}{ |\xi-z|^{11/6}} \frac {dz}{Q_m(|\xi-z|)} 
 \\ &\le 4 \int_{\frac 12(1-|\xi|)}^{1} \left(  \int_{A_s} \frac {s^{1/6}}{(1-|\xi + se^{i\phi}|)^{1/6} } d\phi \right) \frac {ds}{sQ_m(s)}
 \\ & = 4 \int_{\frac 12(1-|\xi|)}^{1} \left(  \int_{A_s} {(s^{-1}-|s^{-1}\xi + e^{i\phi}|)^{-1/6} } d\phi \right) \frac {ds}{sQ_m(s)},
\end{align*}
with 
\[
A_s:=\{\phi\in(0,2\pi]\,\big|\, |\xi + se^{i\phi}|<1\}= \{\phi\in(0,2\pi]\,\big|\, |s^{-1}\xi + e^{i\phi}|<s^{-1}\}.
\]
It is not difficult to see that the inside integral is maximized when $s=1-|\xi|$ (i.e., $(0,2\pi]\setminus A_s$ is a single point) for any $|\xi|\in [1-2\delta^3,1)$, in which case the integrand is bounded above by $[\frac 12 (1-\cos(\phi-\tht_\xi))]^{-1/6} = [\sin \frac 12(\phi-\tht_\xi)]^{-1/3}$ because $\delta\le \frac 1{10^3}$.  But then the inside integral is bounded above by $2\int_0^\pi \left(\frac\phi\pi\right)^{-1/3}d\phi=3\pi$.  Hence \eqref{3.9} holds with $C=12\pi$.

Next consider the case  $\alpha\in[1,\frac 83],$ and define the functions   $g(z) :=\min \{ \frac{1}{|\xi-z|}, \frac 2{1-|\xi|} \}$ and 
\[
f(z):=\min \left\{ \frac{(1-|z|)^{5/6}}{|\xi-z|^{-\alpha+17/6}Q_m(|\xi-z|)}, \frac{2^{-\alpha+17/6}(1-|z|)^{5/6}}{(1-|\xi|)^{-\alpha+17/6}Q_m( \frac 12(1-|\xi|))} \right\},
\]
 as well as  $H_\xi' := B(\xi,\delta^3)\cap \bbD \supseteq H_\xi$.  We can now use Jensen's inequality, \eqref{3.10}, and Lemma~\ref{FoldingLemma} to bound the left-hand side of \eqref{3.9} above by
\begin{align*}
\int_{H_\xi} \frac{(1-|z|)^{5/6}}{|\xi-z|^{17/6}} & \left( \frac{1}{\beta(I)}\int_{I} \frac{|e^{i\tht}-\xi|}{|e^{i\tht}-z|} d\beta(\tht) \right)^{\alpha} \frac {dz}{Q_m(|\xi-z|)} 
 \\ & \le \int_{H_\xi} \frac{(1-|z|)^{5/6}}{|\xi-z|^{17/6}}  \left( 1+ \frac{1}{\beta(I)}\int_{I}  \frac{|\xi-z|}{|e^{i\tht}-z|} d\beta(\tht) \right)^{\alpha} \frac {dz}{Q_m(|\xi-z|)}
 \\ & = \int_{H_\xi} \frac{(1-|z|)^{5/6}}{|\xi-z|^{-\alpha+17/6}}  \left( \frac 1{|\xi-z|}+ \frac{1}{\beta(I)}\int_{I}  \frac 1{|e^{i\tht}-z|} d\beta(\tht) \right)^{\alpha} \frac {dz}{Q_m(|\xi-z|)}
 \\ & \le  \int_{H_\xi'} f(z) \left( g(z)+ \frac{1}{\beta(I)}\int_{I}  \frac 1{|e^{i\tht}-z|} d\beta(\tht) \right)^{\alpha} dz 
 \\ & \le  \int_{H_\xi'} f(z) \left( g(z)+  \frac 1{|e^{i\tht_\xi}-z|} \right)^{\alpha} dz 
 \\ & \le \frac {3^{5/6} 2^\alpha\pi}{ Q_m ( \frac 12 (1-|\xi|) )} +  \int_{H_\xi} \frac{(1-|z|)^{5/6}}{|\xi-z|^{-\alpha+17/6}}  \left( \frac 1{|\xi-z|}+  \frac 1{|e^{i\tht_\xi}-z|} \right)^{\alpha} \frac {dz}{Q_m(|\xi-z|)}
   \\ & \le24\pi  + 4 \int_{H_\xi} \frac{(1-|z|)^{5/6}}{|\xi-z|^{-\alpha+17/6}}  \left( \frac 1{|\xi-z|^\alpha}+  \frac 1{|e^{i\tht_\xi}-z|^\alpha} \right) \frac {dz}{Q_m(|\xi-z|)}. 
   \end{align*}
   Notice that Lemma \ref{FoldingLemma} applies because $\frac 2\pi m(\delta^3)\le  \frac 2\pi m(2\delta)\le \frac 16\le \frac{17}6-\alpha$ shows that $s^{-\alpha+17/6}Q_m(s)$ is increasing on $(0,\delta^3]$.
Using again $1-|z|\le 3|\xi-z|$  for  $z\in H_\xi$ 
yields
\[
\int_{H_\xi} \frac{(1-|z|)^{5/6}}{|\xi-z|^{17/6}}   \frac {dz}{Q_m(|\xi-z|)} \le 3^{5/6} \int_{H_\xi} \frac 1{|\xi-z|^{2}} \frac {dz}{Q_m(|\xi-z|)}   \le 6\pi  \int_{\frac 12(1-|\xi|)}^{1} \frac {ds}{sQ_m(s)},
\]
and then we also have with $H^*:=B(e^{i\tht_\xi},\frac {1-|\xi|}2)\cap\bbD$,
\begin{align}
\int_{H_\xi\setminus H^*} \frac{(1-|z|)^{5/6}}{|\xi-z|^{-\alpha+17/6} |e^{i\tht_\xi}-z|^\alpha} \frac {dz}{Q_m(|\xi-z|)}   
& \le 3^\alpha \int_{H_\xi\setminus H^*} \frac{(1-|z|)^{5/6}}{|\xi-z|^{17/6}}  \frac {dz}{Q_m(|\xi-z|)}  \lb{3.9v}
\\ & \le 162\pi  \int_{\frac 12(1-|\xi|)}^{1} \frac {ds}{sQ_m(s)}. \notag
   \end{align}
Finally, from $1-|z|\le |e^{i\tht_\xi}-z|$, $\alpha\le\frac 83$, and $Q_m\ge 1$ on $[0,1]$ we obtain
\[
\int_{H^*} \frac{(1-|z|)^{5/6}}{|\xi-z|^{-\alpha+17/6} |e^{i\tht_\xi}-z|^\alpha}  \frac {dz}{Q_m(|\xi-z|)} \le \left( \frac {1-|\xi|}2 \right)^{\alpha-17/6} \int_{H^*} { |e^{i\tht_\xi}-z|^{-\alpha+5/6} } dz \le 
12\pi.
\]
This proves \eqref{3.9} with $C=672\pi \le \frac {10^5 -1}{6e^2}$.

Finally, to obtain $C<500$, we perform the following adjustments to the above argument.  We choose $\delta>0$ so that $\beta([\tht -2\delta,\tht +2\delta]) \leq 1.05\pi$ for all $\tht  \in \bbR$, so we always have $\alpha\in[0,2.1]$.  The 1 in  \eqref{3.9x} and \eqref{3.9y} can be replaced by an arbitrary positive constant by lowering $\delta$ further.  Similarly the 2 in \eqref{3.6} can be replaced by an arbitrary constant greater than 1, and the power $\frac 16$ in \eqref{3.9z} by an arbitrarily small positive power (which allows us to turn the $3^{1/6}$ in \eqref{3.9z} into an arbitrary constant greater than $1$; this power then also propagates through the rest of the proof).  This means that the constant in \eqref{3.9z} with the new power can be made arbitrarily close to 1.  The right-hand side of \eqref{3.9w} can be multiplied by an arbitrarily small positive constant if we replace the upper bound in the second integral by a large multiple of $1-|\xi|$ instead of $2(1-|\xi|)$ (which is again possible when $\delta>0$ is small enough), so it follows that it suffices to prove \eqref{3.9} with some $C<500$.  Since in \eqref{3.9v} we can actually replace $3^\alpha$ by $(\sqrt 5)^\alpha \le 5^{1.05}< 5.5$, we indeed obtain \eqref{3.9} with $C=4 (6\pi+33\pi)<500$.  While further lowering of $C$ is possible, we do not do so here.
\end{proof}


\section{Proof of Theorem \ref{T.1.2}} \lb{S4}

Let  $\Omega\subseteq\bbR^2$ be a regulated open bounded Lipschitz domain with $\partial\Omega$ a Jordan curve.  Also assume that $\Omega$ is symmetric with respect to the real axis, $0\in\partial\Omega$, and $(1-\eps,1)\times\{0\}\subseteq\Omega$ for some $\eps>0$.  Let $\Omega^\pm:=\Omega\cap(\bbR\times\bbR^\pm)$ and $\Omega^0:=\Omega\cap(\bbR\times\{0\})$ (these are obviously all simply connected).  Then there is a Riemann mapping  $\calT:\Omega\to\bbD$ with $\calT(\Omega^0)=(-1,1)$ and $\calT(0)=1$, and therefore also $\calT(\Omega^\pm)=\bbD^\pm:=\bbD\cap(\bbR\times\bbR^\pm)$.  Assume also that there are $\beta_\calT,\tilde\beta_\calT$ as in {\bf (H)}, and $\tilde\beta_\calT$ has bounded variation.
Then $\calI(z),\calJ(z)$ from the last section are the integrals
		\begin{align}
		\mathcal{I}(z) &= \frac{2}{\pi} \int_{(-\pi,\pi]} \ln |e^{i\tht}-z|\, d\beta_\calT(\tht ), \notag
		\\ \mathcal{J}(z) &= \frac{2}{\pi} \int_{(-\pi,\pi]} \ln|e^{i\tht}-z|\, d \tilde\beta_\calT(\tht), \lb{3.3''}
		\end{align}
where we replaced integration over $(0,2\pi]$ by $(-\pi,\pi]$ for convenience, and the second formula follows similarly to \eqref{3.3'}.

Given any concave modulus $m$ and $r_0\in(0,\frac 12]$ with $m(2r_0)\le \frac \pi 6$, assume that there are $\Omega$ and $\calT$ as above 
with $\beta_\calT\equiv 0$ on $(-1,1)$ and $\tilde\beta_\calT(\tht) = \frac\pi 2- \frac {\sgn(\tht)}2 \,m(2\min\{|\tht |,r_0\})$ for $\tht\in (-\pi,\pi]$.  Concavity of $m$ then guarantees that $\tilde\beta_\calT$ indeed has modulus of continuity $m$.
Notice also that $d\tilde\beta_\calT(\tht)= -\chi_{(-r_0,r_0)} m'(2|\tht|)d\tht$ on $(-\pi,\pi]$, as well as $|\beta_\calT|=2\pi + m(2r_0)\le 7$.
We show at the end of this section that such $\Omega$ and $\calT$ do exist for any $m$ and $r_0\in(0,\frac 12]$ with $m(2r_0)\le \frac \pi 6$.  


We will first show that if $\int_0^1 \frac {ds}{q_m(s)}<\infty$ and  $x\in\Omega^0$, then the trajectory $X_t^x$ for the stationary weak solution $\omega:=\chi_{\Omega^+}-\chi_{\Omega^-}$ to the Euler equations on $\Omega$ will reach $0\in \partial \Omega$ in finite time.  This will prove Theorem \ref{T.1.2}(i).

Due to symmetry, the particle trajectories $X_t^x$ for this solution coincide with those for the stationary solution $\omega\equiv 1$ on $\Omega^+$.  We will therefore now employ the Biot-Savart law on $\Omega^+$.   Let $\calR:\bbD^+\to\bbD$ be a Riemann mapping with $\calR(1)=1$, so that $\calT^+:=\calR\calT:\Omega^+\to\bbD$ is a Riemann mapping with $\calT^+(0)=1$.  The (time-independent) Biot-Savart law for $\omega\equiv 1$ on $\Omega^+$ can therefore be written as
\beq \lb{4.1}
u(x) = D\calT^+(x)^T  \int_{\Omega^+}  \nabla_{\xi}^\perp G_\bbD(\calT^+(x),\calT^+(y)) \, dy,
\eeq
with $G_\bbD(\xi,z)=\frac 1{2\pi}\ln\frac{|\xi-z|}{|\xi-z^*| |z|}$ the Dirichlet Green's function for $\bbD$.  If $x\in\Omega^0\subseteq\partial\Omega^+$, we have $\calT^+(x)\in\partial \bbD$, where $G_\bbD(\cdot,z)$ vanishes for any fixed $z\in\bbD$ (and $G_\bbD(\cdot,z)<0$ on $\bbD$),  so 
\[
\nabla_{\xi}^\perp G_\bbD(\calT^+(x),\calT^+(y)) = | \nabla_{\xi} G_D(\calT^+(x),\calT^+(y))| \calT^+(x)^\perp.
\]
This suggests one to evaluate 
\[
D\calT^+(x)^T \calT^+(x)^\perp = D\calT^+(x)^T (\det D\calT^+(x))^{-1/2} D\calT^+(x)(1,0),
\]
where $(1,0)$ is the counterclockwise unit tangent to $\Omega^+$ at $x\in\Omega^0$, and we used that the action of the matrix $D\calT^+(x)$ is just multiplication by a complex number with magnitude $\sqrt{\det D\calT^+(x)}$.
Since $D\calT^+$ is of the form $\begin{pmatrix} a & b \\ -b & a \end{pmatrix}$, we have 
\[
D\calT^+(x)^T D\calT^+(x)=(\det D\calT^+(x)) I_2,
\]
 so \eqref{4.1} for $x\in\Omega^0$ becomes
\[
u_1(x) = \sqrt{\det D\calT^+(x)}  \int_{\Omega^+}  |\nabla_{\xi} G_\bbD(\calT^+(x),\calT^+(y))| \, dy \qquad\text{and}\qquad u_2(x)=0.
\]
Since $\Omega^0$ is a smooth segment of $\partial \Omega^+$, standard estimates show that $D\calT^+(x)$ is continuous  and non-vanishing on $\Omega^0$.  Since $\frac d{dt} X_t^x=u(X_t^x)$, it follows that for each $x\in\Omega^0$, the trajectory $X_t^x$ either reaches 0 in finite time or converges to 0 as $t\to\infty$.  It therefore suffices to analyze $u_1(x)$ for $x\in\Omega^0$ close to 0.  

If $x\in\Omega^+\cup\Omega^0$ is not close to the left end of $\Omega^0$, then $\calT(x)\in\overline{\bbD^+}$ is not close to $-1$, so standard estimates yield  $\sqrt{\det D\calR(\calT(x))}\in [c|\calT(x)-1|,c^{-1}|\calT(x)-1|]$ for some $c=c_\calT\in(0,1]$ (because $D\calR(z)\sim z-1$ for $z$ near 1, and $D\calR$ only vanishes at $\pm 1$).  So for all $x\in\Omega^+\cup\Omega^0$ not close to the left end of $\Omega^0$ we have
\beq\lb{4.4}
\sqrt{\det D\calT^+(x)} \left( |\calT(x)-1| \sqrt{\det D\calT(x)} \right)^{-1} \in [c,c^{-1}].
\eeq
From \eqref{3.2x} we have
\beq\lb{4.5}
\det D\calT(x)= \det D\calT(\calT^{-1}(0)) e^{\mathcal{I}(\calT(x))+\mathcal{J}(\calT(x))}.
\eeq
Since $\beta_\calT$ is supported away from $\tht=0$, the term $e^{\mathcal{I}(\calT(x))}$ is  bounded above and below by positive numbers, uniformly in all $x$ that are either close to 0 or not close to $\partial\bbD$.  Moreover, \eqref{3.3''} and the specific form of $\tilde\beta_\calT$ give us for $z\in\bbD$, 
\[
\mathcal{J}(z)
\ge - \frac{4}{\pi} \int_0^{r_0} \ln (|z-1|+\tht ) m'(2\tht )d\tht = -\frac 2\pi m(2r_0) \ln(|z-1|+r_0) + \frac{2}{\pi} \int_0^{r_0} \frac {m(2\tht )}{|z-1|+\tht }d\tht . 
\]
We can now estimate (with a constant $C_{m,r_0}$ changing from one inequality to another)
\begin{align*}
\left| \int_0^{r_0} \frac {m(2\tht )}{|z-1|+\tht }d\tht  - \int_{|z-1|}^{1} \frac {m(r)}{r}dr \right|
& \le \left| \int_{|z-1|/2}^{1/2} \frac {m(2\tht )}{|z-1|+\tht }d\tht  - \int_{|z-1|/2}^{1/2} \frac {m(2\tht )}{\tht }d\tht  \right| + C_{m,r_0}
\\ & \le \left| \int_{|z-1|/2}^{1/2} \frac {|z-1|m(2\tht )}{\tht (|z-1|+\tht )}d\tht  \right| + C_{m,r_0}
\\ & \le \|m\|_{L^\infty} \left| \int_{|z-1|/2}^{1/2} \frac {|z-1|}{\tht ^2}d\tht  \right| + C_{m,r_0}
\\ & \le C_{m,r_0}.
\end{align*}
For $z\in\bbD^0:=\bbD\cap(\bbR\times\{0\})$, we now obtain
\beq\lb{4.3}
\left| \mathcal{J}(z) - \frac 2\pi \int_{|z-1|}^{1} \frac {m(r)}{r}dr \right| \le C_{m,r_0}
\eeq
from this and from an opposite estimate via $\mathcal{J}(z)\le - \frac{4}{\pi} \int_{0}^{r_0} \ln( \frac 12(|z-1|+\tht )) m'(2\tht )d\tht $.  Hence, for a new  $c=c_{\calT,r_0,m}>0$ and all $x\in\Omega^0$ not close to the left end of $\Omega^0$ we obtain
\[
\det D\calT(x) \,Q_m(|\calT(x)-1|)^{-1} \in [c,c^{-1}].
\]
This and \eqref{4.4} show that there is  $c=c_{\calT,r_0,m}>0$ such that  for all $x\in\Omega^0$ close to 0 we have
\[
u_1(x) \ge c |\calT(x)-1|  \sqrt{Q_m(|\calT(x)-1|)}  \int_{\Omega^+}  |\nabla_{\xi} G_\bbD(\calT^+(x),\calT^+(y))| \, dy \qquad\text{and}\qquad u_2(x)=0.
\]

If now $X_t^x\in\Omega^0$ is close to 0 and we let $d(t):=1-|\calT(X_t^x)|=|\calT(X_t^x)-1|$, then
\[
d'(t)= - \left| D\calT(X_t^x) \frac d{dt}X_t^x \right| = - \sqrt{\det D\calT(X_t^x)} u_1(X_t^x)
\]
because $D\calT$ is a multiple of $I_2$ on $\Omega^0$.  Therefore we have (with a new $c>0$)
\beq\lb{4.2}
d'(t) \le - c d(t) Q_m(d(t)) \int_{\Omega^+}  |\nabla_{\xi} G_\bbD(\calT^+(X_t^x),\calT^+(y))| \, dy.
\eeq
Since $|\nabla_{\xi} G_\bbD(\xi,z)|$ is uniformly bounded away from 0 in $(\xi,z)\in \partial\bbD\times\kappa\bbD$ for any fixed $\kappa\in(0,1)$, the integral is bounded below by a positive constant.  But then $d'(t) \le - c q_m(d(t))$, which implies
\[
\int_{d(t)}^1\frac {ds}{q_m(s)} \ge ct + \int_{d(0)}^1\frac {ds}{q_m(s)}
\]
for some $c=c_{\calT,m,r_0}\in(0,1]$.
Since the left-hand side is bounded in $t$ if $\int_0^1 \frac {ds}{q_m(s)}<\infty$, we must have $d(t)=0$ for some $t<\infty$.  This proves that $X_t^x$ reaches $0\in\partial\Omega$ in finite time, and hence Theorem \ref{T.1.2}(i).

This construction also allows us to prove Theorem \ref{T.1.2}(ii).  When $\int_0^1 \frac {ds}{q_m(s)}=\infty$, we can estimate the integral in \eqref{4.2} better after first rewriting it via a change of variables as
\beq\lb{4.6}
\int_{\bbD}  |\nabla_{\xi} G_\bbD(\calT^+(X_t^x),z)| \left[ \det D \calT^+((\calT^+)^{-1}(z)) \right]^{-1}\, dz.
\eeq
Now with $\xi:=\calT^+(X_t^x)$ (and still assuming $X_t^x\in\Omega^0$) we have
\[
 | \nabla_{\xi} G_\bbD(\xi,z)|=\left|\frac{\xi-z}{|\xi-z|^2} - \frac{\xi-z^*}{|\xi-z^*|^2}\right| \ge \frac {10c}{|\xi-z|} \ge \frac {c}{|z-1|}
\] 
for some $c>0$ (which will below change from one inequality to another and may also depend on $\calT,m,r_0$) and all $z\in \bbD\cap(B(1,1)\setminus B(1,|\xi-1|))$ that also lie in the sector  with  vertex $1$, angle $\frac \pi 2$, and axis being the real axis
(call this set  $\mathcal{C_\xi}$ and note that $\mathcal{C}_\xi\subseteq\mathcal{C}_1$).  

If $z\in\mathcal{C}_1$, then for $y:=(\calT^+)^{-1}(z)$ (so $\calT(y)=\calR^{-1}(z)$) we have as above
\[
\det D\calT^+(y) \le c  |\calT(y)-1|^2  Q_m(|\calT(y)-1|)= c |\calT(y)-1| q_m(|\calT(y)-1|).
\]
Indeed, this follows from \eqref{4.4}, \eqref{4.5}, and also \eqref{4.3} for $\calT(y)$ in place of $z$.  The latter extends here even though $\calT(y)\in\calR^{-1}(\mathcal{C}_1)\subseteq\bbD^+$ and so $\calT(y)\notin\bbD^0$ because for some $y$-independent $C>0$ we have $\mathcal{J}(\calT(y))\le - \frac{4}{\pi} \int_{0}^{r_0} \ln( \frac 1C(|\calT(y)-1|+\tht )) m'(2\tht )d\tht $ (recall \eqref{3.3''}).  This in turn is due to the distance of any $v\in \calR^{-1}(\mathcal{C}_1)$ to $\partial\bbD$ being comparable to $|v-1|$, since $\mathcal{C}_1$ has the same property.

So for $z\in\mathcal{C}_\xi$, the integrand in \eqref{4.6} can be bounded below by a multiple of
\[
\frac 1{|z-1||\calR^{-1}(z)-1|q_m(|\calR^{-1}(z)-1|)} \ge  \frac {c^3}{|z-1|^{3/2}q_m(c|z-1|^{1/2})},
\]
with the inequality due to  $|\calR(v)-1|\in [c|v-1|^2,c^{-1}|v-1|^2]$ for all $v\in\overline{\bbD^+}$ as well as $q_m(a^{-1}b)=a^{-1}bQ_m(a^{-1}b)\le a^{-1}b Q_m(b)=a^{-1}q_m(b)$ for $a\in(0,1]$.  The integral is therefore bounded below by a multiple of
\[
\int_{|\xi-1|}^{1}  \frac {dr}{\sqrt r q_m(c\sqrt r)} = \frac 2c \int_{c\sqrt{|\xi-1|}}^{c}  \frac {ds}{ q_m(s)}.
\]
Finally, since $|\xi-1|=|\calR(\calT(X_t^x))-1| \le c^{-1} |\calT(X_t^x)-1|^2 = c^{-1}d(t)^2$, from \eqref{4.2} and $cq_m(c^{-1}d)\le q_m(d)$  for $c\in(0,1]$ and $d\in(0,c]$ we obtain
\[
d'(t) \le - c q_m(d(t)) \left( \int_{c^{-1}d(t)}^{1}  \frac {ds}{ q_m(s)} -C\right) \le - c^2 q_m(c^{-1}d(t)) \left( \int_{c^{-1}d(t)}^{1}  \frac {ds}{ q_m(s)} -C \right)
\]
whenever $X_t^x\in\Omega^0$ is close enough to 0, with some $c=c_{\calT,m,r_0}\in(0,1]$ and $C=C_{\calT,m,r_0}\ge0$.  But dividing this by the right-hand side and integrating yields (with a new $C$)
\[
\ln \int_{d(t)}^{1}  \frac {ds}{ q_m(s)} \ge \ln \left( \int_{c^{-1}d(t)}^{1}  \frac {ds}{ q_m(s)} - C \right) \ge ct + \ln \left( \int_{c^{-1}d(0)}^{1}  \frac {ds}{ q_m(s)} -C \right) \ge ct
\]
for all $t>0$, as long as $x\in\Omega^0$ is close enough to 0  (so the last parenthesis is $\ge 1$).
This now yields Theorem \ref{T.1.2}(ii).

\medskip
\noindent {\bf Construction of a Domain Corresponding to a Given Modulus}
\medskip

We will now show that a domain as above does exist.  We will do this by taking the desired $\bar\beta_\calT=\beta_\calT+\tilde\beta_\calT$ and  obtaining the domain $\Omega:= \calS(\bbD)$ via the corresponding mapping $\calS$ from \eqref{3.2y'}.  Since $\bar\beta_\calT$ has bounded variation, we can now use the equivalent formula
\beq \lb{3.2y}
\mathcal{S}'(z) = \mathcal{S}'(0) \, \exp \left(- \frac{1}{\pi} \int_{(-\pi,\pi]} \ln(1-ze^{-i\tht}) \, d\bar\beta_\calT(\tht) \right)
\eeq
(see \cite[Corollary 3.16]{Pomm}).
Our $\Omega$ will in fact be a perturbed isosceles triangle, with one vertex and the center of the opposite ``side'' on the real axis, and the modulus $m$ will be ``attained'' at the center of that side (where $\Omega$ will therefore be concave).

Given any concave modulus $m$ and $r_0\in(0,\frac 12]$ with $m(2r_0)\le \frac\pi 6$, let us define  $\tilde\beta(\tht):=\frac\pi 2-\frac {\sgn(\tht)}2 \,m(2\min\{|\tht |,r_0\})$ on $(-\pi,\pi]$  (and let its derivative be $2\pi$-periodic).  Then let $\beta$ be such that $\beta(0)=0$ and 
\[
d\beta|_{(-\pi,\pi]}:= \left( \frac{2\pi}3+\pi m_0 \right) \delta_{\pi} + \frac{2\pi}3  \delta_{\pi/3} + \frac{2\pi}3 \delta_{-\pi/3},
\]
where $m_0:=\frac 1{\pi} m(2r_0)$ and $\delta_{\tht _0}$ is the Dirac measure at $\tht =\tht _0$.  Clearly $\bar\beta:=\beta+\tilde\beta$ satisfies $\bar\beta(\pi)=\bar\beta(-\pi)+2\pi$, and $\bar\beta-\frac\pi 2$ is odd on $\bbR$ (which is needed for symmetry of $\Omega$ with respect to the real axis).

We now use \eqref{3.2y} with the choice  $\mathcal{S}'(0):=1$ to define
\[
\mathcal{V}(z): = \exp \left(- \frac{1}{\pi} \int_{(-\pi,\pi]} \ln(1-ze^{-i\tht}) \, d\bar\beta(\tht) \right)=
(1+z^3)^{-2/3}v(z),
\]
where we consider the branch of the logarithm with $\ln:\bbR^+\times\bbR\to \bbR\times(-\frac \pi 2, \frac \pi 2)$ (since ${\rm Re}(1-ze^{-i\tht})>0$), use that $\Pi_{k=0}^2 (1-ze^{-i(2k-1)\pi/3})=1+z^3$, and also define
\[
v(z):= (1+z)^{-m_0} \exp \left( \frac{2}{\pi} \int_0^{r_0} \ln(1-ze^{-i\tht}) m'(2\tht ) \, d\tht  \right).
\]
Since  ${\rm Im} \ln(1-ze^{-i\tht})\in (-\frac \pi 2, \frac \pi 2)$ for all $(z,\tht )\in\bbD\times\bbR$, the imaginary part of the above exponent belongs to $(-\frac \pi 2 m_0, \frac \pi 2 m_0)$.  This and ${\rm Re} (1+z)>0$ now yield 
\[
|\arg v(z)|< \pi m_0=m(2r_0)\le \frac \pi 6
\]
 for all $z\in\bbD$.  Since also $|\arg(1+z^3)|<\frac \pi 2$, it follows that ${\rm Re}\, \mathcal{V}(z)>0$  for all $z\in\bbD$.  But then  the mapping 
$\calS:\bbD\to\bbC$ given by
 \[
 \calS(z):=\int_1^z \mathcal{V}(\xi) \, d\xi 
 \]
is injective, and $\calT:=\calS^{-1}$ is a Riemann mapping for $\Omega:=\calS(\bbD)$ with $\partial\Omega$ is a Jordan curve.  Note that $\Omega$ is bounded because $\mathcal V(z)=O(\sum_{k=0}^2 |e^{i(2k-1)\pi/3}-z|^{-5/6})$.  Since $\mathcal{V}((-1,1))\subseteq\bbR^+$, we have $\calS((-1,1))\subseteq\bbR$, and then $\calS((-1,1))=\Omega^0$, with $\calS(1)=0\in\partial\Omega$ its right endpoint.

Observe that  $\arg(\mathcal{V}(e^{i\phi})) $ is uniformly continuous on $(e^{i(2k-1)\pi/3},e^{i(2k+1)\pi/3})$ for $k=0,1,2$.  This is because the same is  true for the argument of  $(1+e^{3i\phi})^{-2/3}(1+e^{i\phi})^{-m_0}$, while
\[
\arg \left( \mathcal{V}(e^{i\phi})(1+e^{3i\phi})^{2/3}(1+e^{i\phi})^{m_0} \right) = \frac{2}{\pi} \int_0^{r_0} \arg(1-e^{i(\phi-\tht)}) m'(2\tht ) \, d\tht,
\]
 which is continuous in $\phi$ because $m$ is continuous.  We therefore have that for each $\eps>0$ there are points $0=\phi_0<\dots<\phi_N=2\pi$ (with $e^{i(2k-1)\pi/3}$ being among them) and $a_1,\dots,a_N\in\bbR$ such that 
$
|\arg(\calS(e^{i\phi'})-\calS(e^{i\phi})) -a_n | < \eps
$
whenever $\phi_{n-1}<\phi<\phi'<\phi_n$.  Then $\Omega$ is a regulated domain by Theorem 3.14 in \cite{Pomm}.  So it has a unit forward tangent vector from \eqref{1.7a} for each $\tht\in\bbR$, and \eqref{3.2y'} 
shows that with its argument $ \bar \beta_\calT$ from \eqref{1.7b} we have
\beq \lb{3.2yy}
\mathcal{V}(z)=\mathcal{S}'(z) =  \exp \left(\frac{i}{2\pi} \int_{-\pi}^{\pi} \frac{e^{i\tht}+z} {e^{i\tht}-z} \left( \bar \beta_\calT(\tht) -\tht-\frac\pi 2 \right) \, d\tht \right)
\eeq
  because $\calS'(0)=\mathcal{V}(0)=1$.

In the definition of $\mathcal{V}$, we can replace $\bar\beta(\tht)$ by the $2\pi$-periodic function $\bar\beta(\tht)-\tht-\frac\pi 2$
because  $\int_0^{2\pi} \ln(1-ze^{-i\tht})  d\tht=\ln 1 =0$.
Integration by parts then  yields
\begin{align*}
\mathcal{V}(z) & = \exp \left(\frac{i}{\pi} \int_{-\pi}^{\pi} \frac{z} {e^{i\tht}-z}  \left( \bar \beta(\tht) -\tht -\frac\pi 2 \right) \, d\tht \right) 
\\ & = \exp \left(\frac{i}{2\pi} \int_{-\pi}^{\pi} \left( \frac{e^{i\tht}+z} {e^{i\tht}-z}-1\right) \left( \bar \beta(\tht) -\tht -\frac\pi 2 \right) \, d\tht \right).
\end{align*}
From this and \eqref{3.2yy} we find that
\[
\frac 1{2\pi} \int_{-\pi}^{\pi}  \frac{e^{i\tht}+z} {e^{i\tht}-z}  \left( \bar \beta_\calT(\tht) - \bar \beta(\tht) \right) d\tht = \frac 1{2\pi} \int_{-\pi}^{\pi}  \left( \bar \beta(\tht) -\tht -\frac\pi 2 \right) \, d\tht + 2k\pi = 2k\pi
\]
for some  $k\in\bbZ$ and all $z\in\bbD$ (because $\bar\beta(\tht)-\frac\pi 2$ and $\tht$ are odd).  Hence $\bar \beta_\calT - \bar \beta\equiv  2k\pi$, so $\Omega$ and $\calT$ are indeed the domain and Riemann mapping we wanted to construct.

\section{Proof of Lemma \ref{FoldingLemma}} \lb{S5}

Monotone Convergence Theorem shows that it suffices to consider bounded $f,g,h$.  We will prove this via a series of ``foldings'' of $\beta|_I$ 
onto smaller and smaller intervals that shrink toward $\tht^*$.
We will show that at each step the relevant integral cannot decrease.  

Define $\beta^0:=\beta|_I$ and let $\beta^1$ be the measure for which 
\[
\beta^1(A)=
\begin{cases}
\beta^0(A) & \text{if $A\subseteq(-\infty,\tht^*-2\delta)\cup (\tht^*,\infty)$},
\\ 0 &  \text{if $A\subseteq [\tht^*-2\delta,\tht^*-\delta)$},
\\ \beta^0(A\cup(2(\tht^*- \delta)-A)) &  \text{if $A\subseteq [\tht^*-\delta ,\tht^*]$}
\end{cases}
\]
for any measurable $A\subseteq\bbR$. 
That is, we obtain $\beta^1$ from $\beta^0$ by reflecting $\beta^0|_{[\tht^*-2\delta,\tht^*-\delta )}$ across $\tht^*-\delta$ onto $(\tht^*-\delta,\tht^*]$.  In particular, $\beta^1$ is supported on $[\tht^*-\delta ,\tht^*+2\delta]$ and both measures have total mass $\beta(I)$.
We now let 
\[
G^j(z) : = g(z) + \frac{1}{\beta(I)}\int_{I} h(|e^{i\tht}-z|)\, d\beta^j(\tht),
\]
and want to show that
\beq\lb{3.3}
\int_{H} f(z) G^0(z)^{\alpha} dz \leq  \int_{H} f(z)G^1(z)^{\alpha} dz.
\eeq

Let $\tilde H:=\{re^{i\phi}\in H\,|\, \phi\in[\tht^*-\delta-\pi, \tht^*-\delta] \}$ and let $H':=\{re^{i(2(\tht^*-\delta)-\phi)} \,|\, re^{i\phi}\in \tilde H\}$ be its reflection across the line connecting 0 and $e^{i(\tht^*-\delta)}$.  The properties of $H$ ensure that $H'\subseteq H$.  If now $z\in H\setminus \tilde H$, then $|e^{i\tht}-z|\ge |e^{i(2(\tht^*-\delta)-\tht )}-z|$
for any $\tht \in[\tht^*-2\delta,\tht^*-\delta)$.  This and $h$ being non-increasing show that $G^0(z)\le G^1(z)$ for all $z\in H\setminus \tilde H$, and in particular for all  $z\in H\setminus (\tilde H\cup H')$.  To conclude \eqref{3.3}, it hence suffices to show that
\beq\lb{3.4}
f(z)G^0(z)^{\alpha}+ f(z')G^0(z')^{\alpha} \leq f(z)G^1(z)^{\alpha} + f(z')G^1(z')^{\alpha}
\eeq
holds for any $z=re^{i\phi}\in\tilde H$, with $z':=re^{i(2(\tht^*-\delta)-\phi)}\in H'$ its reflection across the line connecting 0 and $e^{i(\tht^*-\delta)}$.

Note that the properties of $f$ and $g$ show that $f(z')\ge f(z)$ and $g(z')\ge g(z)$.  Let
\begin{align*}
b_+ &:= g(z) +  \frac{1}{\beta(I)} \int_{[\tht^*-\delta,\tht^*+2\delta]} h(|e^{i\tht}-z|) \, d\beta^0(\tht) \qquad(\ge 0),
\\ b_- &:= \frac{1}{\beta(I)} \int_{[\tht^*-2\delta,\tht^*-\delta)} h(|e^{i\tht}-z|) \, d\beta^0(\tht) \qquad(\ge 0),
\\ b'_+ &:= g(z') +  \frac{1}{\beta(I)}\int_{[\tht^*-\delta,\tht^*+2\delta]} h(|e^{i\tht}-z'|) \, d\beta^0(\tht) \qquad(\ge 0),
\\ b'_- &:= \frac{1}{\beta(I)}\int_{[\tht^*-2\delta,\tht^*-\delta)} h(|e^{i\tht}-z'|) \, d\beta^0(\tht) \qquad(\ge 0).
\end{align*} 
Then $G^0(z)=b_++b_-$, $G^0(z')=b'_++b'_-$, $G^1(z)=b_++b'_-$, and $G^1(z')=b'_++b_-$, so 
\[
G^0(z)+G^0(z')=G^1(z)+G^1(z').
\]
We also have $b_+'\ge b_+$ and $b'_-\le b_-$ due to $g(z')\ge g(z)$, $h$ being non-increasing, and the definition of $z'$.  This implies
\[
0\le G^1(z)\le \min\{G^0(z),G^0(z')\}\le \max\{G^0(z),G^0(z')\} \le G^1(z').
\]
The last two relations, 
together with convexity of the function $x^\alpha$ on $[0,\infty)$, now yield
\[
G^0(z)^{\alpha}+ G^0(z')^{\alpha} \leq G^1(z)^{\alpha} + G^1(z')^{\alpha}.
\]
From this and  $(f(z')- f(z))(G^1(z')^\alpha-G^0(z')^\alpha)\ge 0$ we obtain \eqref{3.4}, and therefore \eqref{3.3}.

An identical (modulo reflection) argument shows that if  $\beta^2$ is obtained from $\beta^1$  by reflecting $\beta^1|_{(\tht^*+\delta,\tht^*+2\delta ]}$  across $\tht^*+\delta$ onto $[\tht^*,\tht^*+\delta)$, then we have
\[
\int_{H} f(z) G^1(z)^{\alpha} dz \leq  \int_{H} f(z)G^2(z)^{\alpha} dz.
\]
We can then repeat this with $\frac\delta 2$ in place of $\delta$ because $\beta^2$ is supported on $[\tht^*-\delta,\tht^*+\delta]$ and has total mass $\beta(I)$.  Continuing in this way, we obtain a sequence of measures $\beta^0,\beta^2,\beta^4,\dots$, each $\beta^{2j}$ having total mass $\beta(I)$ and supported on $[\tht^*-2^{1-j}\delta,\tht^*+2^{1-j}\delta]$, such that
\[
\int_{H} f(z) G^{2j}(z)^{\alpha} dz \leq  \int_{H} f(z)G^{2(j+1)}(z)^{\alpha} dz
\]	
for $j=0,1,\dots$.	
Since 
the integrands are uniformly bounded and converge pointwise to $f(z) (g(z) + h(|e^{i\tht^*}-z|) )^{\alpha}$ as $j\to\infty$, Dominated Convergence Theorem finishes the proof.

\end{document}